\documentclass[12pt]{amsart}

\pdfoutput=1
\usepackage{latexsym}
\usepackage[centertags]{amsmath}
\usepackage{amsfonts}
\usepackage{amssymb}
\usepackage{amsthm}
\usepackage{newlfont}
\usepackage{graphics}
\usepackage{color}
\usepackage{booktabs}
\usepackage{wrapfig}
\usepackage{subfigure}
\usepackage{graphicx}
\usepackage{extarrows}
\usepackage{diagbox}
\usepackage{wrapfig}
\usepackage[nocompress]{cite}
\usepackage[usenames,dvipsnames]{xcolor}
\usepackage{graphicx}
\usepackage[latin1]{inputenc} 

\makeatletter
\newcommand{\tabcaption}{\def\@captype{table}\caption}
\makeatother

\newtheorem{theo}{Theorem}[section]
\newtheorem{defn}[theo]{Definition}
\newtheorem{exam}[theo]{Example}
\newtheorem{lem} [theo]{Lemma}
\newtheorem{cor}[theo]{Corollary}
\newtheorem{prop}[theo]{Proposition}
\newtheorem{rem}[theo]{Remark}

\newtheorem{conj}{Conjecture}

\newtheorem{fact}[theo]{Fact}
\makeatletter \@addtoreset{equation}{section}
\@addtoreset{theo}{}\makeatother

\numberwithin{equation}{section}
\numberwithin{theo}{section}

\def\N{\mathbb{N}}

\def\x{\mathbf{x}}
\def\P{\mathbb{P}}
\def\a{{\pmb{\alpha}}}
\def\b{{\pmb{\beta}}}

\def\N{\mathbb{N}}

\def\R{\mathbb{R}}
\def\Q{\mathbb{Q}}
\def\diag{\mathop{\mathrm{diag}}}

\title[Kekul\'{e} numbers conjectures by Chebyshev polynomials]{Proving some conjectures on Kekul\'{e} numbers for certain benzenoids by using Chebyshev polynomials}

\author{Guoce Xin$^{1,*}$ and Yueming Zhong$^{2}$}

 \address{ $^{1,2}$School of Mathematical Sciences, Capital Normal University,
 Beijing 100048, PR China}
\email{$^1$\texttt{guoce\_xin@163.com}\ \& $^2$\texttt{zhongyueming107@gmail.com}}

\date{ \today}
\thanks{$*$ This work was partially supported by NSFC(12071311).}

\begin{document}

\begin{abstract}
In chemistry, Cyvin-Gutman enumerates  Kekul\'{e} numbers for certain benzenoids and record it as $A050446$ on OEIS. This number is exactly the
two variable array
$T(n,m)$ defined by the recursion $T(n, m) = T(n, m-1) + \sum^{\lfloor\frac{n-1}{2}\rfloor}_{k=0} T(2k, m-1)T(n-1-2k, m)$, where $T(n,0)=T(0,m)=1$ for
all nonnegative integers $m,n$. Interestingly, this number also appeared in the context of weighted graphs, graph polytopes, magic labellings, and unit primitive
matrices, studied by different authors. Several interesting conjectures were made on the OEIS. These conjectures are related to
both the row and column generating function of $T(n,m)$. In this paper, give explicit formula of the column generating function, which is also
the generating function $F(n,x)$ studied by B\'{o}na, Ju, and Yoshida. We also get trig function representations by using Chebyshev polynomials of the second kind. This allows us to prove all these conjectures.
\end{abstract}

\maketitle

\vspace{-5mm}
\maketitle

\noindent
\begin{small}
 \emph{Mathematic subject classification}: Primary 05A15; Secondary 15A18, 05C78, 52B11.
\end{small}

\noindent
\begin{small}
\end{small}

\section{Introduction}\label{sec:introduction}
Throughout this paper, we use standard set notations $\N, \P$ for nonnegative integers, and positive integers, respectively.
We also use $\chi(true)=1$ and $\chi(false)=0$.

\begin{defn}
  The two parameter sequence $A050446$ on \cite{OEIS} is defined by
  \begin{equation}\label{eq:T-floor}
T(n, m) = T(n, m-1) + \sum^{\lfloor\frac{n-1}{2}\rfloor}_{k=0} T(2k, m-1)T(n-1-2k, m), \qquad  m\in \P
\end{equation}
with initial condition $T(n,0)=1$ for all $n\in \N$.
\end{defn}
Table \ref{tab:Tnm} gives the first several values of $T(n,m)$ as the $(n,m)$ entries.

\begin{table}[htbp]
\setlength{\tabcolsep}{0.3mm}{
\begin{tabular}{c|rrrrrrrrrrr}
  \hline
  \diagbox{$n$}{$m$} &0 & 1 & 2 & 3    & 4      & 5      & 6         &7       & 8      & 9  &...\\
  \hline
  0 &1 & 1     & 1      & 1   & 1       & 1     & 1      &1       &1      &1      &...\\
  1 &1 & 2     & 3      & 4   & 5       & 6     & 7      &8       &9      &10      &...\\
  2 &1 & 3     & 6      & 10  & 15      & 21    & 28     &36      &45     &55      &... \\
  3 &1 & 5     & 14     & 30  & 55      & 91    & 140    &204     &285    &385     &... \\
  4 &1 & 8     & 31     & 85  & 190     & 371   & 658    &1086    &1695   &2530    &... \\
  5 &1 & 13    & 70     & 246 & 671     & 1547  & 3164   &5916    &10317  &17017   &... \\
  \vdots & \vdots  &\vdots &\vdots &\vdots &\vdots &\vdots &\vdots &\vdots &\vdots & \\
  \hline
\end{tabular}}\caption{The first values of $T(n,m)=K\{z(n,m)\}.$
}\label{tab:Tnm}
\end{table}

One can find many interesting properties about $T(n,m)$ on \cite{OEIS}. It counts certain chemistry structure, and
several combinatorial structures as we shall introduce. Let
$$ T(x,m)=\sum_{n\ge 0} T(n,m) x^n, \qquad T(n,y)= \sum_{m\ge 0} T(n,m) y^m$$
be the row and column generating function of $T(n,m)$, respectively.
There are six conjectures related to $T(x,m)$ and $T(n,y)$.
We will state them in Section 2 after introducing some notations.

\subsection{Kekul\'{e} structures}
The number $T(n,m)$ appeared in chemistry \cite{Gutman-Kekule} as the number $K\{z(n,m)\}$ of Kekul\'{e} structures of the benzenoid hydrocarbon $z(n,m)$,
where $z(n,m)$ is a zigzag chain interpreted as $n$ tier condensed linear chains (rows) of $m$ hexagons each in a zigzag arrangement as in Figure \ref{fig:ZigzagChain}.
\begin{figure}[!ht]
\centering{
\includegraphics[height=1.9 in]{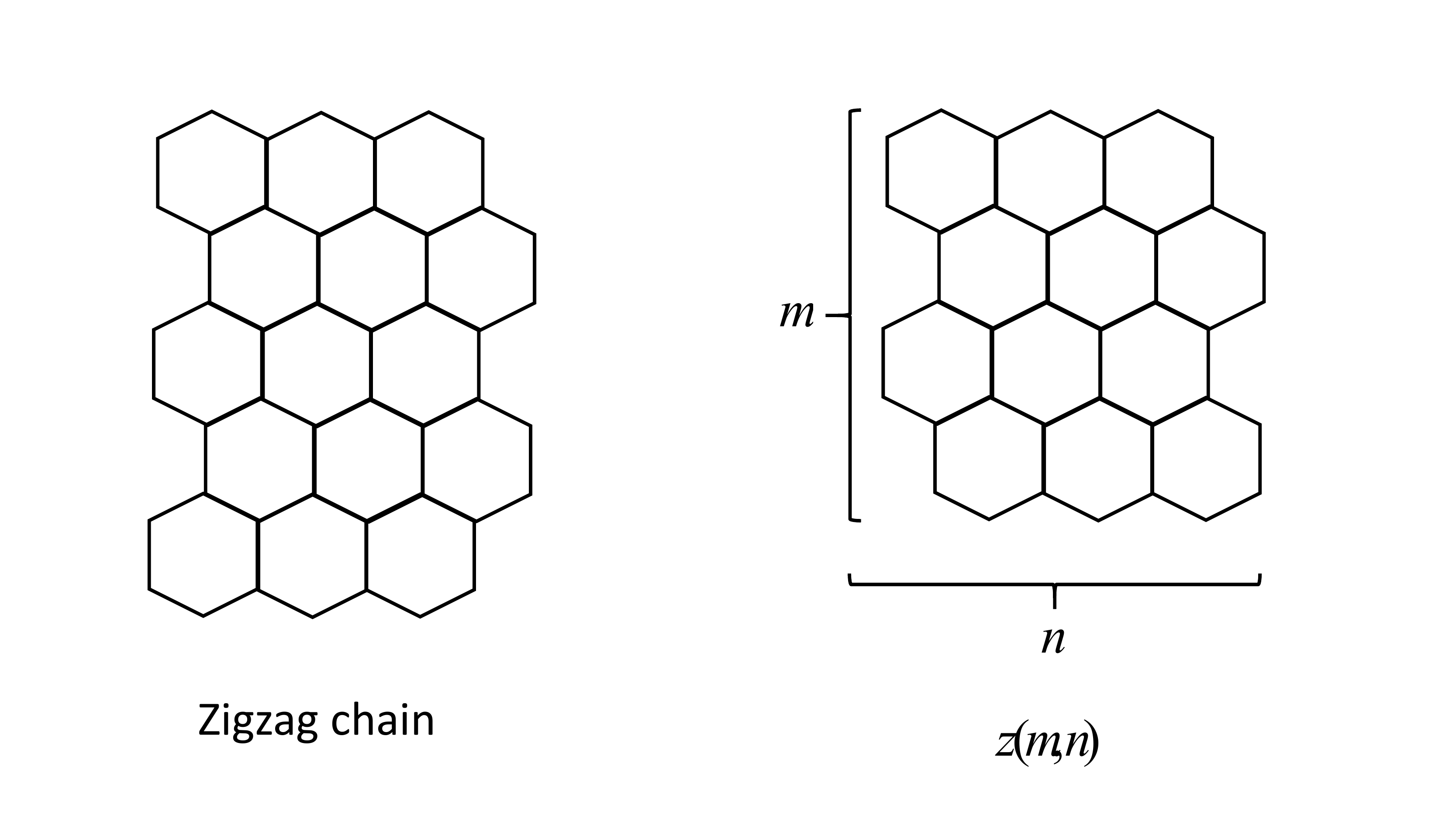}}\vspace{-0.5cm}
\caption{Zigzag chain and $z(m,n)$
}\label{fig:ZigzagChain}
\end{figure}

The special case $z(1,1)$ is the single structure of benzene (\normalfont{$C_6H_6$}). In Figure \ref{fig:benzene},
picture (1) illustrates $C_6H_6$, and picture (2) gives the
two Kekul\'{e} structures regarded differently in chemistry.
\begin{figure}[!ht]
\centering{
\includegraphics[height=1.5 in]{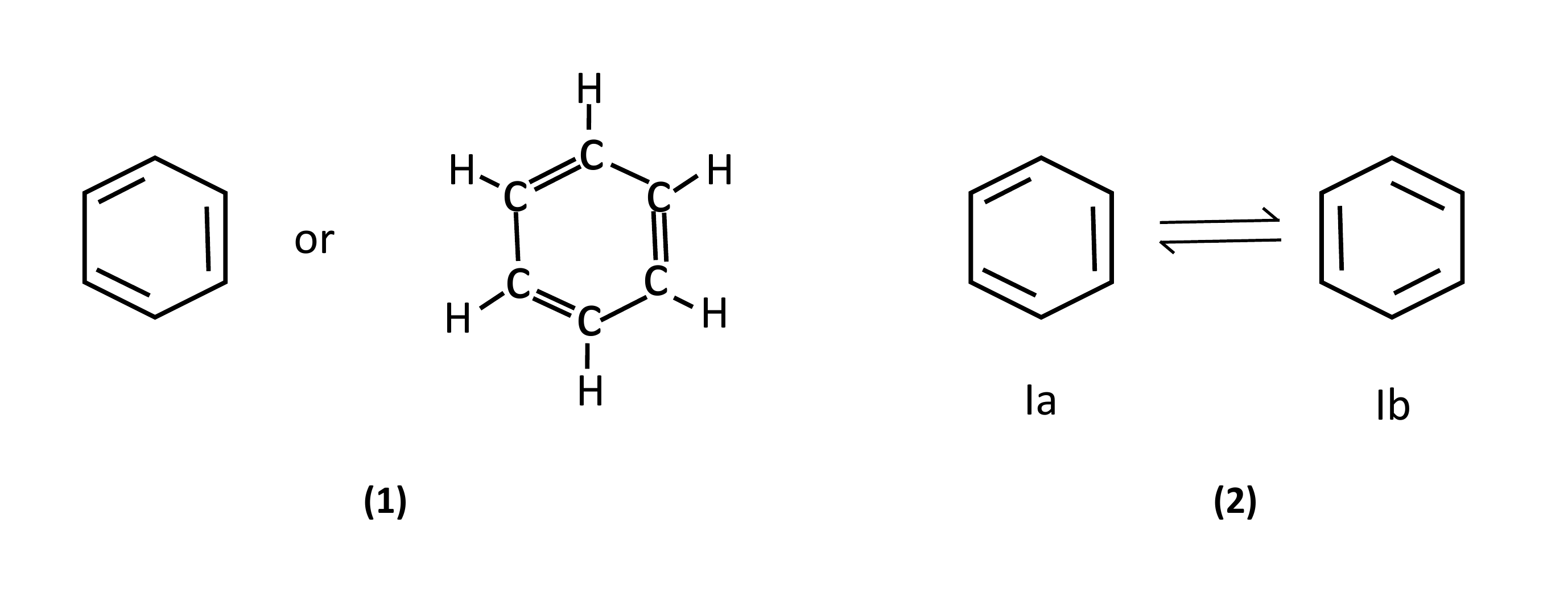}}\vspace{-0.5cm}
\caption{Structural formula of benzene  and its two Kekul\'{e} structures
}\label{fig:benzene}
\end{figure}

The explicit definition of $K\{z(n,m)\}$ is too involved so we only give several of them in Figure \ref{fig:Kznumber}. We cannot find
a reference proving that $K\{z(n,m)\}=T(n,m)$.
\begin{figure}[!ht]
\centering{
\includegraphics[height=1.7 in]{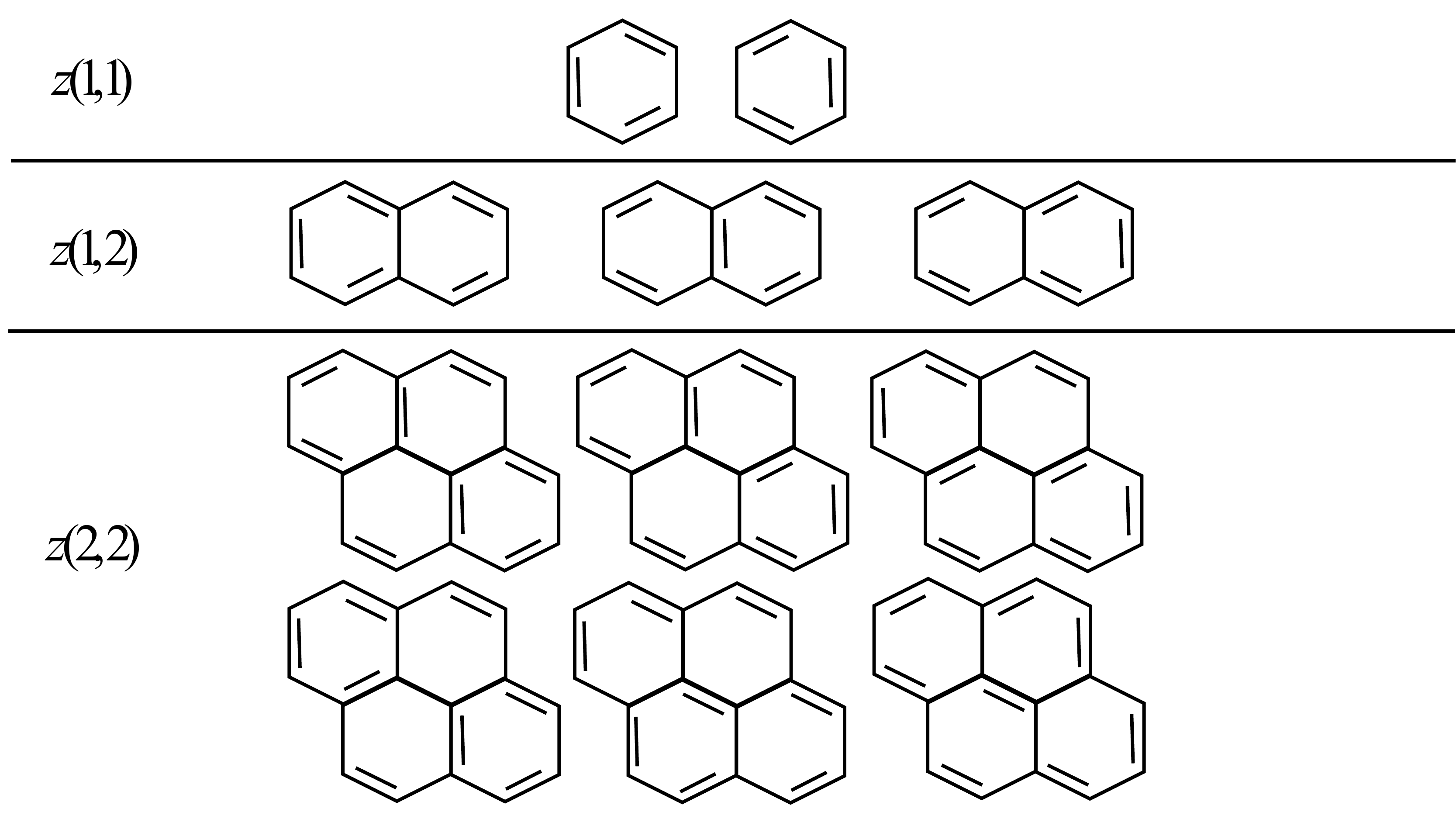}}\vspace{-0.5cm}
\caption{The representation of the Kekul\'{e} structures of  $K\{z(1,1)\}$,  $K\{z(1,2)\}$ and  $K\{z(2,2)\}$.
}\label{fig:Kznumber}
\end{figure}

\subsection{Combinatorial Models}
Let $G=(V,E)$ be a simple graph with vertex set $V=\{v_1,v_2,...,v_n\}$ and edge set $E$. We will introduce
three combinatorial models: i) Graph polytope studied in \cite{Lee,Stanley-Ehrhart,Graph-Polytope};  ii) Weighted graph studied in \cite{Bona}; iii) Magic labellings of a graph
studied in \cite{Stanley-reciprocityTheorem}.
Though the concepts are for general graphs, we will focus on the linear graph (or path) $L_n$ with $n$ vertices.

The \emph{graph polytope} $P(G)$ associated to a simple graph $G$ is defined as
$$P(G)= \{ (x_1,\dots, x_n)\in [0,1]^n \mid v_iv_j \in E \Rightarrow x_i+x_j\le 1\}.$$
The dilation of $P(G)$ by $m$ is denoted $mP(G)$. The Ehrhart series of $P(G)$ is defined
by
$$ Ehr(G)=Ehr(G,x)= 1+\sum_{m\ge 1} |(mP(G)\cap \N^n)| x^m. $$
See \cite{Lee,Stanley-Ehrhart} for further references.

A \emph{a weighted graph} of $G$ with distribution $\alpha=(m_1,m_2,...,m_n)\in \N^n$ is a triplet $WG_{\alpha}=(V,E,\alpha)$,
where the weight of $v_i$ is $m_i$. It is natural to define the weight of an edge $v_iv_j\in E$ to be $m_i+m_j$.
Let $WG(m)$ be the number of all weighted graphs of $G$ with a fixed upper bound $m$
for weight of each vertex and edge of $G$. That is, $m_i\in [m]:=\{0,1,2,...,m\}$ for all $i$ and
\begin{equation}\label{eq:weight}
 m_i+m_j\le m \text{\ if\ } v_iv_j \in E.
\end{equation}
An interesting question is to compute the generating function of $WG(m)$:
$$\rho(G)=\rho(G,x)=\sum\limits_{m=0}^{\infty}WG(m)x^m.$$

Let $G=(V,E)$ be a finite (undirected) graph with vertex set $V=\{v_1,\dots,v_n\}$ and edge set $E$. We allow loops in $E$ so $G$
is also called a pseudo graph in the terminology of \cite{Frank}. A magic labelling of $G$ with magic sum $s$
is a labelling of the edges in $E$ by nonnegative integers such that for each vertex $v\in V$,
the weight $wt(v)$ of $v$, defined to be the sum of labels of all edges incident to $v$, is equal to $s$.
Denote by $M_G(s)$ the number of magic labelling of $G$ with magic sum $s$.
More precisely, $M_G(s)$ counts the number of maps $\mu: E\mapsto \N$ satisfying
\begin{align}
  \label{e-magic-sum}
  wt(v_i):= \sum_{v_iv_j\in E }   \mu(v_i,v_j) =s, \qquad i=1,2,\dots, n.
\end{align}
An interesting question is compute the generating function
$$\mathcal{M}(G)=\mathcal{M}(G,x) = \sum_{s\ge 0} M_G(s) x^s. $$

Finally, let $A_m=(a_{i,j})$ be the order $m$ matrix defined by $a_{i,j}=\chi(i+j\le m+1)$. It is called a \emph{unit-primitive} matrix
by Jeffery \cite{Jeffery} in a different context.

The following result is claimed on \cite{OEIS}, but we cannot find a real proof in the literature.
\begin{theo}\label{t-4interpretations}
For $(n,m)\in \N^2$, the number $T(n,m)$ has the following interpretations.
\begin{enumerate}
  \item The $(1,1)$ entry of the matrix $A_{m+1}^{n+1}$, also equal to $u_{m+1}^TA_{m+1}^{n-1}u_{m+1}$, where $u_m=(1,1,\dots,1)^T\in \N^m$.

 \item The number $WL_n(m)$ of weighted graphs of $L_n$ with a fixed upper bound $m$.

 \item The number of integer lattice points in the dilated graph polytope $mP(L_n)$.

 \item The number of magic labellings of the pseudo graph $\tilde{L}_n$ with
 magic sum $m$, where $\tilde{L}_n$ is obtained from the linear graph $L_{n+1}$ by attaching
 one loop at each vertex.
\end{enumerate}
Consequently in terms of generating functions, we have
\begin{align*}
  T(n,x)=\sum_{m\ge 0} T(n,m) x^m = Ehr({L_n})=\rho(L_n)=\mathcal{M}(\tilde{L}_n).
\end{align*}
\end{theo}
The equivalence of (2), (3), (4) can be easily seen by an example of $m=3$.
By definition, $WL_3(m)$ is the number of $\N$-solutions $\alpha=(m_1,m_2,m_3)$ of the system
$m_1+m_2\leq m,\ m_2+m_3\leq m$. The $\alpha$'s are exactly the lattice points of the $m$ dilated graph polytope
$P(L_3)=\{(x_1,x_2,x_3)\in \R: x_1+x_2\leq 1, \ x_2+x_3 \leq 1\}$.
Let $\ell_1=m-m_1, \ell_2= m-m_1-m_2, \ell_3=m-m_2-m_3, \ell_4=m-m_3$. Then the system
is translated to
$$ m_1+\ell_1=m,\ \ell_2+m_1+m_2=m,\ \ell_3+m_2+m_3=m,\ \ell_4+m_3 =m,$$
which says the labelling as in Figure \ref{fig:L3L3tilde} is a magic labeling of $\tilde{L}_3$.

\begin{figure}[!ht]
\centering{
\includegraphics[height=1.2 in]{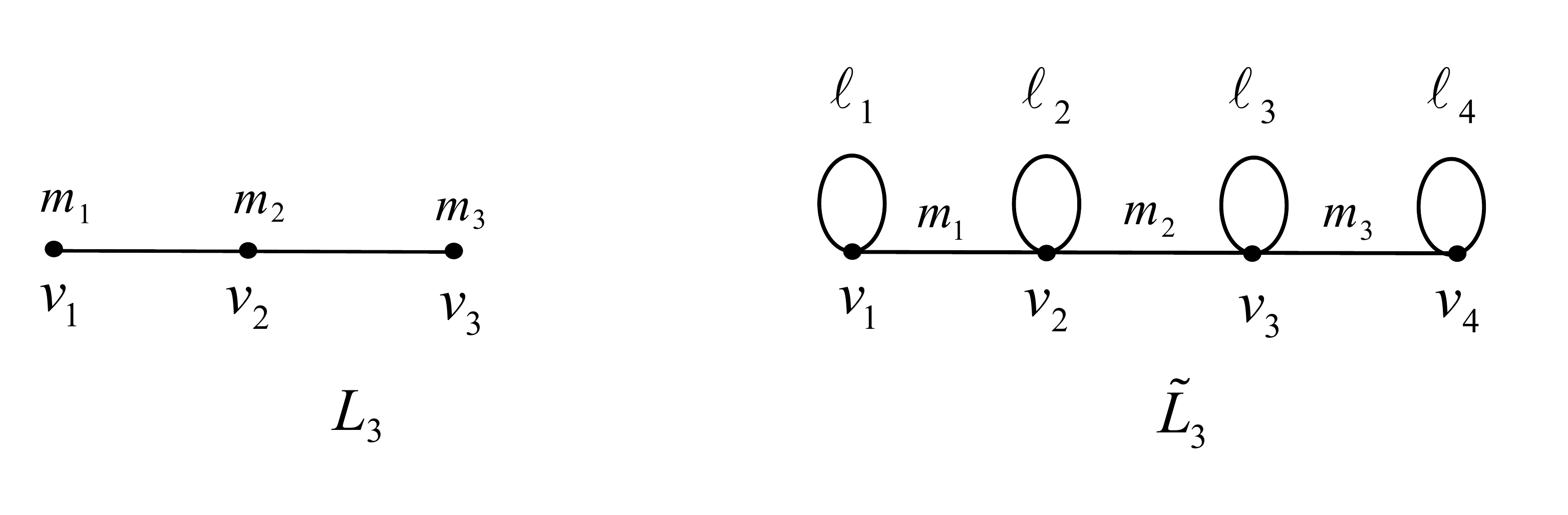}}\vspace{-0.5cm}
\caption{Weighted graph of $L_3$ and magic labelling of $\tilde{L}_3$.
}\label{fig:L3L3tilde}
\end{figure}

To see that (2) is equivalent to (1), we describe $\alpha$ as a walk on the finite state $[m]$ of the $\alpha_i$'s:
first choose $\alpha_1\in [m]$,
then choose $\alpha_2$ to be any number in $[m-\alpha_1]$, since we have the restriction $\alpha_1+\alpha_2\le m$,
and then choose $\alpha_3$ to be any number in $[m-\alpha_2]$ by the restriction $\alpha_2+\alpha_3\le m$. The transfer matrix
from the state of $\alpha_i$ to the state of $\alpha_{i+1}$ is just $A_{m+1}$, the unit primitive  matrix of order $m+1$.
It then follows that $WL_3(m)=u_{m+1}^T A_{m+1}^{2}u_{m+1}$.

The transfer matrix idea was used in \cite{Bona}, where weighted graphs were systematically enumerated and the following beautiful formula
about $L_n$ was cited.
\begin{theo}\label{theo:Bona}
(B\'{o}na et al. \cite{Bona-2}). For $m=0,1,2,...$, let
$$F(m,x)=\sum\limits_{k=0}^{\infty}WL_k(m)x^k=1+\sum\limits_{k=0}^{\infty}(u_{m+1}^TA_{m+1}^ku_{m+1})x^{k+1}.$$
Then
\begin{align*}
  F(m,x) &=\frac{1}{-x+F(m-1,-x)}=\frac{1}{-x+\frac{1}{x+F(m-2,x)}} \\
         &:=[-x,x,F(m-2,x)],
\end{align*}
where $F(0,x)=1/(-x+1)=[-x,1]$ and $F(1,x)=(1+x)/(1-x-x^2)=[-x,x,1]$. That is, $F(m,x)=[-x,x,-x,x,...,(-1)^{m-1}x,1]$.
\end{theo}
However, the cited paper was never finished (private communication), so it is necessary to give a proof.

\subsection{Main Results}
Our objectives in this paper is 4-folded.

Firstly, we give a proof of Theorem \ref{t-4interpretations}. For convenience, we use $v_{m,i}$ to denote the $i$-th unit column vector in $\R^m$. Let $\overline{T}(n,m)=u_m^TA_m^nv_{m,1}=u_m^TA_m^{n-1}u_m$. Then it is sufficient to prove the following.
      \begin{theo}\label{theo:T-Tbar}
      For $m,n\in \N$, we have
      \begin{equation}\label{eq:T-Tbar}
      T(n,m)=\overline{T}(n,m+1).
      \end{equation}
      \end{theo}

Secondly, we give two proofs of Theorem \ref{theo:Bona}. Further more, we obtain trig function representation.
\begin{theo}\label{theo:Fnxsin}
For $n\in \P$, we have $$F(n,x)=(-1)^{n+1}\frac{\sin (n+1) \theta-\sin n\theta}{\sin (n+2)\theta -\sin (n+1)\theta}=(-1)^{n+1}\frac{\cos (\frac{2n+1}{2}\theta) }{\cos(\frac{2n+3}{2}\theta)},$$ where $\cos \theta=\frac{(-1)^{n}x}{2}$.
\end{theo}

Thirdly, we obtain a factorization of the characteristic polynomial $f_{A_n}$ of $A_n$ as follows, where throughout the paper,
we use $f_{A}(x)=|xI_n-A|$ to denote the characteristic polynomial of an order $n$ matrix $A$.

  \begin{theo}\label{theo:fAn}
For a given positive integer $n$, we have $$f_{A_n}(x)=|xI_n-A_n|=\prod\limits_{j=1}^{n}\left[x-(-1)^{n+1}\left(2\cos\frac{(2j-1)\pi}{2n+1}\right)^{-1}\right].$$
\end{theo}

Finally, we prove the six conjectures related to $T(n,m)$ in Section \ref{sec:conjectures}.

\subsection{Summary}
This paper is organized as follows:
In Section \ref{sec:conjectures} we state the six conjectures after introducing some notations. Chebyshev polynomial of the second kind
plays an important role in our development.
Section \ref{sec:result-Tbar} proves Theorem \ref{theo:T-Tbar}, an equivalent form of Theorem \ref{t-4interpretations}. The idea is to use matrix
computation to derive certain recursion of $\overline{T}_{n,m}$. The idea is used to give a proof of Theorems \ref{theo:Bona},
but seems irrelevant to the other sections. Section \ref{sec:result-An} studies many properties of the unit primitive matrix $A_n$, especially
its eigenvalues. As consequences, we give the proof of Theorems \ref{theo:Fnxsin}-\ref{theo:fAn} and Conjecture \ref{conj:fAn}. We also
give explicit formula of the column generating function  $T(x,m)=F(m,x)$.
In Section \ref{sec:result-T}, we first use the magic labelling model to sketch a proof of Conjectures \ref{conj:M-rho}-\ref{conj:M-symmetric},
which are related to the row generating function $T(n,x)$. Conjectures \ref{conj:M-equation}-\ref{conj:M-limit} are about
the numerator of $T(n,x)$. We prove them in later subsections by using the explicit formulas of the column generating function $T(x,m)$.

\section{Six Conjectures Related to $T(n,m)$}\label{sec:conjectures}
We will introduce 6 interesting conjectures related to the unit-primitive matrix  $A_n$ and the sequence $T(m,n)$.
Their proofs will be given in later sections.

Our development highly relies on \emph{Chebyshev polynomial of the second kind} (CPS for short), which is defined recursively by
\begin{align*}
  U_0(x)=1, \quad U_1(x)=2x,  \quad U_n(x)=2xU_{n-1}(x)-U_{n-2}(x).
\end{align*}
We also need the well-known formula (see, e.g., \cite{Chebyshev-Polynomial-Second}):
\begin{equation}\label{eq:Un-equation}
U_n(x)=\frac{\sin[(n+1)\cos^{-1}x]}{\sin(\cos^{-1}x)}=\sum\limits_{k=0}^{\lfloor \frac{n}{2} \rfloor}(-1)^k\binom{n-k}{k}(2x)^{n-2k}.
\end{equation}

Many results are related to the polynomials $P_n(x), n\in \P$ defined by
\begin{align}
  \label{e-Pnk}
 P_n(x)=\sum_{k=0}^{n}P(n,k)x^k = \sum_{k=0}^{n} (-1)^{\lfloor\frac{3k}{2}\rfloor}\binom{\lfloor\frac{n+k}{2}\rfloor}{k} x^k.
\end{align}

The first conjecture is related to the unit-primitive matrix $A_n$.
\begin{conj}\label{conj:fAn}
(A187660 \cite{OEIS}). (i) For positive integer $n \in \P$. The characteristic polynomial $f_{A_n}(x)$ of the $n \times n$ unit-primitive matrix $A_n$ is:
\begin{align}\label{e-fAn}
  f_{A_n}(x) = \det( x I_{n} -A_n) =x^nP_n(1/x).
\end{align}

(ii) The $n$ eigenvalues of $A_n$ are given by $w_{n,j}=U_{n-1}(\delta_{n,j})$, where $\delta_{n,j}=2\cos(\frac{2j-1}{2n+1}\pi)$ for $j = 1,...,n$.
\end{conj}

The other 5 conjectures are all related to the sequence $T(n,m)$ defined by \eqref{eq:T-floor}.

Let $E_n$ be the Euler number \cite[Section 3.16]{Stanley-book} with exponential generating function
$$ \sec(x)=\sum_{n\ge 0} \frac{E_{2n}}{(2n)!} x^{2n},\qquad  \tan(x)=\sum_{n\ge 0} \frac{E_{2n+1}}{(2n+1)!} x^{2n+1}.$$

The following result was observed in \cite{OEIS}.
\begin{lem}
For fixed nonnegative integer $n$, $T(n,m)$ is a polynomial in $m$ of degree $n$ with leading coefficient $\ell_n=E_n/n!$. Therefore
the $n$-th row generating function of $T(n,m)$ is of the following form
  $$ T(n,x)=\sum_{m\ge 0} T(n,m) x^m =  H_{n}(x)/(1-x)^{n+1},$$
where $H_n(x)$ is a polynomial of degree no more than $n$.  Here we start with row $0$.
\end{lem}
\begin{proof}
The reason is so simple that we include it here. For the first part,
let $p_n(m)=T(n,m)$. We need to show that $p_n(x) \in \Q[x]$ with $\deg p_n(x)=n$ and the leading coefficient is $E_n/n!$.
The assertion clearly holds for $n=0$, since $p_0(m)=1$. Assume
the lemma holds for all $k<n$. To show the assertion holds for $p_n(m)$, we rewrite
\eqref{eq:T-floor} as
\begin{align*}
\Delta(p_n(m))= p_n(m)-p_n(m-1) =\sum^{\lfloor\frac{n-1}{2}\rfloor}_{k=0} p_{2k}(m-1)p_{n-1-2k}(m).
\end{align*}
That is, the divided difference of $p_n(m)$ is a sum of polynomials of degree $n-1$
with leading coefficient
$$L=\sum^{\lfloor\frac{n-1}{2}\rfloor}_{k=0}\ell_{2k}\ell_{n-1-2k}=\sum^{\lfloor\frac{n-1}{2}\rfloor}_{k=0}\frac{E_{2k}}{(2k)!} \frac{E_{n-1-2k}}{(n-1-2k)!} $$ by the induction hypothesis. Therefore $p_n(m)$ is a polynomial in $m$ of degree $n$ with leading coefficient $\ell_n=L/n $. It is in fact $E_n$
follows by equating coefficient of $x^{n-1}$ in the identity $\frac{\mathrm{d}}{\mathrm{d}x} (\sec(x)+\tan(x))=\sec(x)(\sec(x)+\tan(x)).$

The second part follows by classical theory on rational functions. See \cite{Stanley-book}.
\end{proof}

Simple calculation shows that $H_0(x) = H_1(x) = 1$. The second conjecture gives more details on $H_n(x)$.
\begin{conj}\label{conj:M-rho}
(A205497 \cite{OEIS}).
For $n\ge 2$, the numerator of the $n$-th row generating function of $T(n,m)$ is a polynomial of degree $n-2$.
That is
$$ H_{n}(x)=M_{n-2,0}+M_{n-3,1}x+M_{n-4,2}x^2+\cdots+M_{0,n-2}x^{n-2},$$
where $M_{n-2-j,j}$ is as in Table \ref{tab:Mnk} below for $j=0,1,\dots,n-2$.
\end{conj}

The array $M_{i,j}$ is given by the sequence A205497 in \cite{OEIS}. Christopher H. Gribble kindly calculated the first 100 antidiagonals of the array $M$ which starts as in Table \ref{tab:Mnk}.

\begin{table}[h]
\begin{tabular}{c|lllllll}
  \hline
  \diagbox{$n$}{$k$} & 0 & 1 & 2 & 3 & 4 & 5 & ...\\
  \hline
  0 & 1 & 1  & 1     & 1      & 1       & 1        &...\\
  1 & 1 & 3  & 7     & 14     & 26      & 46       &... \\
  2 & 1 & 7  & 31    & 109    & 334     & 937      &... \\
  3 & 1 & 14 & 109   & 623    & 2951    & 12331    &... \\
  4 & 1 & 26 & 334   & 2951   & 20641   & 123216   &... \\
  5 & 1 & 46 & 937   & 12331  & 123216  & 1019051  &... \\
  \vdots & \vdots  &\vdots &\vdots &\vdots &\vdots &\vdots  & \\
  \hline
\end{tabular}\caption{The $M_{n,k}$
}\label{tab:Mnk}
\end{table}

There are several conjectures about $M_{i,j}$.

\begin{conj}\label{conj:M-symmetric}
(A205497 \cite{OEIS}). $M_{n,k}=M_{k,n}$, for all $n$ and $k$; that is, $M$ is symmetric about the central terms $\{1,3,31,623,...\}$.
\end{conj}

%
%
%

\begin{conj}\label{conj:M-equation}
(A205497 \cite{OEIS}). Assuming Conjecture \ref{conj:M-rho}, we have these following results.
\begin{itemize}
  \item [(1)] The generating function for row (or column) $n$ of $M$ is of the form
\begin{equation}\label{eq:M-equation-conj}
\frac{G_n(x)}{\left(P_1(x))^{n+1} (P_2(x))^n  ... (P_n(x)\right)^2  P_{n+1}(x)}.
\end{equation}
  \item [(2)] The numerator $G_n(x)$ is a polynomial of degree $\frac{n(n+4)(n+5)}{6}$($A005586(n)$, see \cite{OEIS}).
  \item [(3)] The denominator $(P_1(x))^{n+1} (P_2(x))^n  ... (P_n(x))^2  P_{n+1}(x)$ is a polynomial of degree $\frac{n(n+1)(n+2)}{6}$($A000292(n+1)$, see \cite{OEIS}).
\end{itemize}
\end{conj}

\begin{conj}\label{conj:M-example}
(A205497 \cite{OEIS}). The generation functions from column 0 to 4 of $M$ are the following 5 in order.
\begin{itemize}
  \item [(0)] $1/(1-x)$;
  \item [(1)] $\frac{1}{(1-x)^2(1-x-x^2)}$;
  \item [(2)] $\frac{1 - x^2 - x^3 - x^4 + x^5}{(1-x)^3(1-x-x^2)^2(1 - 2x - x^2 + x^3)}$;
  \item [(3)] $\frac{1 + x - 6x^2 - 15x^3 + 21x^4 + 35x^5 - 13x^6 - 51x^7 + 3x^8 + 21x^9 + 5x^{10} + x^{11} - 5x^{12} - x^{13} - x^{14}}{(1-x)^4(1-x-x^2)^3(1 - 2x - x^2 + x^3)^2(1 - 2x - 3x^2 + x^3 + x^4)}$;
  \item [(4)] $\frac{N_4(x)}{(1-x)^5(1-x-x^2)^4(1 - 2x - x^2 + x^3)^3(1 - 2x - 3x^2 + x^3 + x^4)^2(1 - 3x - 3x^2 + 4x^3 + x^4 - x^5)}$.
\end{itemize}
Here
\begin{tiny}
\begin{align*}
N_4(x)&=1 + 4x - 31x^2 - 67x^3 + 348x^4 + 418x^5 - 1893x^6 - 1084x^7+ 4326x^8 + 4295x^9 - 7680x^{10} - 9172x^{11} + 9104x^{12}\\
       &+ 11627x^{13}- 5483x^{14} - 10773x^{15} + 1108x^{16} + 7255x^{17} + 315x^{18} - 3085x^{19}- 228x^{20} + 669x^{21}+ 102x^{22}\\
        &- 23x^{23} - 45x^{24} - 16x^{25} + 11x^{26}+ 2x^{27} - x^{28}.
\end{align*}
\end{tiny}
\end{conj}

\begin{conj}\label{conj:M-limit}
(A205497 \cite{OEIS}). Assuming Conjecture \ref{conj:M-rho}, we have $$\mathrm{lim}_{n\rightarrow\infty}\frac{M_{n+1,m}}{M_{n,m}} = U_m\left(\cos\left(\frac{\pi}{2m+3}\right)\right).$$ This limit is also equal to spectral radius of the $m \times m$ unit-primitive matrix  $A_{m}$, where identical limits for the columns of the transpose $M^T$ of $M$.
\end{conj}

\section{Proofs of Theorems \ref{theo:Bona} and \ref{theo:T-Tbar}}\label{sec:result-Tbar}
In this section, we first prove Theorem \ref{theo:T-Tbar} by using matrix computations.
Then we use similar idea to give our first proof of Theorem \ref{theo:Bona}. The ideas in this section
are irrelevant to the other sections.

We use the following notations: $X_m=(v_{m,1},v_{m,2},...,v_{m,m})$,
 $\bar{A}_m=A_m-X_m$ and $J_m=u_mu_m^T$ (the all $1$ matrix). The basic properties of these
 order $m$ matrices are given in the following lemma. They will be used in
 the next subsection.
\begin{lem}\label{lem:A}
For a given positive integer $m$, the following equations all hold true.
\begin{itemize}
  \item [(1)] $\bar{A}_mX_m+X_mA_m=J_m$ and $X_m\bar{A}_m+A_mX_m=J_m$;
  \item [(2)] $A_mv_{m,1}=u_m$;
  \item [(3)] $u_m^TX_m=u_m^T$;
  \item [(4)] $A_mv_{m,1}u_m^T=J_m$;
  \item [(5)] $\bar{A}_mX_mv_{m,1}=0$;
  \item [(6)] $A_mX_mv_{m,1}=v_{m,1}$;
  \item [(7)] $u_m^T\bar{A}_mv_{m,1}=\bar{u}_{m}^T\bar{A}_{m}v_{m,1}=u_{m-1}^TA_{m-1}v_{m-1,1}$, where $\bar{u}_m=u_m-v_{m,m}$.
\end{itemize}
\end{lem}
\begin{proof}
The first six equalities are straightforward, so we only prove the seventh equality. By using block matrix operation, we have
\begin{align*}
  u_m^T\bar{A}_mv_{m,1} &=(u_{m-1}^T,1)\begin{pmatrix}
                                    A_{m-1} & 0 \\
                                     0      & 0
                                  \end{pmatrix}\binom{v_{m-1,1}}{0}\\
                     &=(u_{m-1}^TA_{m-1},0)\binom{v_{m-1,1}}{0}\\
                     &=(u_m^T-v_{m,m}^T)\bar{A}_mv_{m,m}=\bar{u}_m^T\bar{A}_mv_{m,1}\\
                     &=u_{m-1}^TA_{m-1}v_{m-1,1}.
\end{align*}
\end{proof}

\subsection{Proof of Theorem \ref{theo:T-Tbar}}
To prove Theorem \ref{theo:T-Tbar}, it is sufficient to show that
$\overline{T}(n,m+1)$ satisfy the same recursion \eqref{eq:T-floor} as for $T(n,m)$.
This is Corollary \ref{cor:proof-T-floor} below.

We need some lemmas.
\begin{lem}\label{lem:baseTeo}
Given three positive integers $m$, $n$ and $k$ with $1 \le k\le n-2$, we have
\begin{equation}\label{eq:baseTeo}
\bar{A}_m^kA^{n-k}_m-\bar{A}^{k+2}_mA_m^{n-k-2}=\bar{A}_m^{k+1}v_{m,1}u_m^TA_m^{n-k-2}.
\end{equation}
\end{lem}
\begin{proof}
\begin{align*}
&\bar{A}_m^kA^{n-k}_m-\bar{A}^{k+2}_mA_m^{n-k-2}\\
&=\bar{A}_m^{k}(\bar{A}_m+X_m)A_m^{n-k-1}-\bar{A}^{k+2}_mA_m^{n-k-2}\\
&=\bar{A}_m^{k+1}(\bar{A}_m+X_m)A_m^{n-k-2}+\bar{A}_m^{k}X_mA_m^{n-k-1}-\bar{A}^{k+2}_mA_m^{n-k-2}\\
&=\bar{A}_m^{k+2}A_m^{n-k-2}+\bar{A}_m^{k}(\bar{A}_mX_m+X_mA_m)A_m^{n-k-2}-\bar{A}^{k+2}_mA_m^{n-k-2}\\
\text{(by Lemma \ref{lem:A}(1))}&=\bar{A}_m^{k}J_mA_m^{n-k-2}\\
\text{(by Lemma \ref{lem:A}(4))}&=\bar{A}_m^{k}A_mv_{m,1}u_m^TA_m^{n-k-2}\\
&=\bar{A}_m^{k}(\bar{A}_m+X_m)v_{m,1}u_m^TA_m^{n-k-2}\\
&=\bar{A}_m^{k+1}v_{m,1}u_m^TA_m^{n-k-2}+\bar{A}_m^{k}X_mv_{m,1}u_m^TA_m^{n-k-2}\\
\text{(by Lemma \ref{lem:A}(5))}&=\bar{A}_m^{k+1}v_{m,1}u_m^TA_m^{n-k-2}.
\end{align*}
\end{proof}

\begin{lem}\label{lem:formula1-2}
For given two positive integers $m$ and $n$, we have
$$u_m^T\bar{A}_m^{2\lfloor\frac{n-1}{2}\rfloor+1}A_m^{n-2\lfloor\frac{n-1}{2}\rfloor-1}v_{m,1}=u_m^T\bar{A}_m^nv_{m,1}.$$
\end{lem}
\begin{proof}
We discuss by the parity of $n$ as follows.

If $n$ is odd, then $2\lfloor \frac{n-1}{2}\rfloor + 1=n$. We have
$$u_m^T\bar{A}_m^{2\lfloor\frac{n-1}{2}\rfloor+1}A_m^{n-2\lfloor\frac{n-1}{2}\rfloor-1}v_{m,1}=u_m^T\bar{A}_m^nv_{m,1}.$$

If $n$ is even, then $2\lfloor \frac{n-1}{2}\rfloor + 1=n-1$. We have
\begin{align*}
  u_m^T\bar{A}_m^{2\lfloor\frac{n-1}{2}\rfloor+1}A_m^{n-2\lfloor\frac{n-1}{2}\rfloor-1}v_{m,1}
  &=u_m^T\bar{A}_m^{n-1}A_mv_{m,1}\\
  &=u_m^T\bar{A}_m^{n-1}(\bar{A}_m+X_m)v_{m,1}\\
  \text{(by Lemma \ref{lem:A}(5))}&=u_m^T\bar{A}_m^{n}v_{m,1}.
\end{align*}
This completes the proof.
\end{proof}

\begin{lem}\label{lem:formula1-lemma}
For given two positive integers $m$ and $n$, we have
$$\bar{A}_mA_m^{n-1}-\bar{A}_m^{2\lfloor\frac{n-1}{2}\rfloor+1}A_m^{n-2\lfloor\frac{n-1}{2}\rfloor-1}+v_{m,1}u_m^TA_m^{n-1}=\sum\limits_{i=0}^{\lfloor\frac{n-1}{2}\rfloor}\bar{A}_m^{2i}v_{m,1}u_m^TA_m^{n-1-2i}.$$
\end{lem}
\begin{proof}
By adding Equation \eqref{eq:baseTeo} with respect to $k=1,3,...,2\lfloor\frac{n-1}{2}\rfloor-1$, we get
\begin{equation}\label{eq:shang1}
\bar{A}_mA_m^{n-1}-\bar{A}_m^{2\lfloor\frac{n-1}{2}\rfloor+1}A_m^{n-2\lfloor\frac{n-1}{2}\rfloor-1}=\sum\limits_{i=1}^{\lfloor\frac{n-1}{2}\rfloor}\bar{A}_m^{2i}v_{m,1}u_m^TA_m^{n-1-2i}.
\end{equation}
This is just a equivalent form of the lemma.
\end{proof}

A direct consequence of Lemma \ref{lem:formula1-lemma} is the following.
\begin{theo}\label{theo:formula1}
Let $m$ and $n$ be two positive integers. Then for any
vectors $\a=(\alpha_1,....\alpha_m)$ and $\b=(\beta_1,...,\beta_m)$, we have
$$\a^T\bar{A}_mA_m^{n-1}\b-\a^T\bar{A}_m^{2\lfloor\frac{n-1}{2}\rfloor+1}A_m^{n-2\lfloor\frac{n-1}{2}\rfloor-1}\b+\a^Tv_{m,1}u_m^TA_m^{n-1}\b=\sum\limits_{i=0}^{\lfloor\frac{n-1}{2}\rfloor}\a^T\bar{A}_m^{2i}v_{m,1}u_m^TA_m^{n-1-2i}\b.$$
\end{theo}

In particular, setting $\a=u_m$ and $\b=v_{m,1}$ gives the following corollary.
\begin{cor}\label{cor:proof-T-floor}
For positive integers $m$ and $n$, we have
\begin{equation}\label{eq:Tbar-recursion}
\overline{T}(n,m)=\overline{T}(n,m-1)+\sum\limits_{k=0}^{\lfloor \frac{n-1}{2} \rfloor}\overline{T}(2k,m-1)\overline{T}(n-2k-1,m).
\end{equation}
\end{cor}
\begin{proof}
It is easy to check that Equation \eqref{eq:Tbar-recursion} holds for $n=1$. Now for $n\ge 2$, we have
\begin{align*}
   \overline{T}(n,m)&=u_m^TA_m^nv_{m,1}\\
                   &=u_m^T(\bar{A}_m+X_m)A_m^{n-1}v_{m,1}\\
                   &=u_m^T\bar{A}_mA_m^{n-1}v_{m,1}+u_m^TX_mA_m^{n-1}v_{m,1}\\
                   &=u_m^T\bar{A}_mA_m^{n-1}v_{m,1}+u_m^TA_m^{n-1}v_{m,1}\\
\text{(by Theorem \ref{theo:formula1})} &=u_m^T\bar{A}_m^{2\lfloor\frac{n-1}{2}\rfloor+1}A_m^{n-2\lfloor\frac{n-1}{2}\rfloor-1}v_{m,1}+\sum\limits_{i=0}^{\lfloor\frac{n-1}{2}\rfloor}u_m^T\bar{A}_m^{2i}v_{m,1}u_m^TA_m^{n-1-2i}v_{m,1}\\
\text{(by Lemma \ref{lem:formula1-2})}&=\overline{T}(n,m-1)+\sum^{\lfloor\frac{n-1}{2}\rfloor}_{k=0} \overline{T}(2k, m-1)\overline{T}(n-1-2k, m).
\end{align*}
\end{proof}

\subsection{Proof of Theorem \ref{theo:Bona}}
We will prove a more general result by using similar technique as in the previous subsection.

\begin{lem}\label{lem:baseeq}
For given three positive integers $m$, $n$ and $k$, we have
$$A_m^{n-k}X_m\bar{A}_m^k+A_m^{n-k-1}X_m\bar{A}_m^{k+1}=A_m^{n-k}v_{m,1}u_m^T\bar{A}_m^k.$$
\end{lem}
\begin{proof}
We have
\begin{align*}
   A_m^{n-k}X_m\bar{A}_m^k+A_m^{n-k-1}X_m\bar{A}_m^{k+1}
   &=A_m^{n-k-1}(A_mX_m+X_m\bar{A}_m)\bar{A}_m^k \\
   \text{(by Lemma \ref{lem:A}(1))}&=A_m^{n-k-1}J_m\bar{A}_m^k \\
   &=A_m^{n-k-1}u_mu_m^T\bar{A}_m^k \\
   \text{(by Lemma \ref{lem:A}(2))}&=A_m^{n-k}v_{m,1}u_m^T\bar{A}_m^k.
\end{align*}
\end{proof}

\begin{lem}\label{lem:Newformula1}
For given two positive integers $m$ and $n$, we have
$$A_m^nX_m+(-1)^{n-1}X_m\bar{A}_m^n=\sum\limits_{k=0}^{n-1}(-1)^kA_m^{n-k}v_{m,1}u_m^T\bar{A}_m^k.$$
\end{lem}
\begin{proof}
By Lemma \ref{lem:baseeq}, we have
\begin{equation}\label{eq:baseeq}
  (-1)^k(A_m^{n-k}X_m\bar{A}_m^k+A_m^{n-k-1}X_m\bar{A}_m^{k+1})=(-1)^kA_m^{n-k}v_{m,1}u_m^T\bar{A}_m^k.
\end{equation}
Adding Equation \eqref{eq:baseeq} with respect to $k=0,...,n-1$ gives
$$A_m^nX_m+(-1)^{n-1}X_m\bar{A}_m^n=\sum\limits_{k=0}^{n-1}(-1)^kA_m^{n-k}v_{m,1}u_m^T\bar{A}_m^k.$$
This completes the proof.
\end{proof}

\begin{lem}\label{lem:Newformula2}
For given two positive integers $m$ and $n$, we have
\begin{equation}\label{eq:AT-newformula2}
A_m^nX_m-(-1)^{n}X_m\bar{A}_m^n+(-1)^{n}v_{m,1}u_m^T\bar{A}_m^n-A_m^nv_{m,1}v_{m,m}^T=\sum\limits_{k=0}^{n}(-1)^kA_m^{n-k}v_{m,1}\bar{u}_m^T\bar{A}_m^k.
\end{equation}
\end{lem}
\begin{proof}
By Lemma \ref{lem:baseeq}, we have
\begin{equation}\label{eq:baseeq}
  (-1)^k(A_m^{n-k}X_m\bar{A}_m^k+A_m^{n-k-1}X_m\bar{A}_m^{k+1})=(-1)^kA_m^{n-k}v_{m,1}u_m^T\bar{A}_m^k.
\end{equation}
Adding Equation \eqref{eq:baseeq} with respect to $k=0,...,n-1$ gives
\begin{align*}
 A_m^nX_m+(-1)^{n-1}X_m\bar{A}_m^n&=\sum\limits_{k=0}^{n-1}(-1)^kA_m^{n-k}v_{m,1}u_m^T\bar{A}_m^k \\
(\text{by } u_m^T\bar{A}_m=\bar{u}_m^T\bar{A}_m)\quad &= A_m^n v_{m,1} u_m^T+ \sum\limits_{k=1}^{n-1}(-1)^kA_m^{n-k}v_{m,1}\bar{u}_m^T\bar{A}_m^k\\
&=A_m^{n}v_{m,1}\bar{u}_m^T+A_m^{n}v_{m,1}v_{m,m}^T + \sum\limits_{k=1}^{n-1}(-1)^kA_m^{n-k}v_{m,1}\bar{u}_m^T\bar{A}_m^k.
\end{align*}
The lemma then follows.
\end{proof}

\begin{theo}\label{theo:base1}
Let $m$ and $n$ be positive integers. Then for any vectors $\a=(\alpha_1,...,\alpha_m)^T$ and $\b=(\beta_1,...,\beta_m)^T$, we have
\begin{scriptsize}
\begin{equation}\label{eq:alphabeta-AT}
\a^T A_m^nX_m\b-(-1)^{n}\a^TX_m\bar{A}_m^n\b+(-1)^{n}\alpha_1u_m^T\bar{A}_m^n\b-\a^T A_m^nv_{m,1}\beta_m=\sum\limits_{k=0}^{n}\a^T A_m^{n-k}v_{m,1}\bar{u}_m^T\bar{A}_m^k\b(-1)^k.
\end{equation}
\end{scriptsize}
\end{theo}
\begin{proof}
Multiply Equation \eqref{eq:AT-newformula2} from the left by $\a$ and from the right by $\b$.
\end{proof}

\begin{cor}\label{cor:Tnm}
For $n,m \in \P$ and $n \ge 2$, we have
$$\overline{T}(n-1,m)=\sum\limits^{n}_{k=0} (-1)^{n-k}\overline{T}(k, m)\overline{T}(n-k, m-1).$$
\end{cor}
\begin{proof}
Setting $\a=u_m$ and $\b=v_{m,1}$ (so $\b_m=0$) in Theorem \ref{theo:base1} gives
$$
u_m^T A_m^nX_mv_{m,1}-(-1)^{n}u_m^TX_m\bar{A}_m^nv_{m,1}+(-1)^{n}u_m^T\bar{A}_m^nv_{m,1}=\sum\limits_{k=0}^{n}u_m^T A_m^{n-k}v_{m,1}u_m^T\bar{A}_m^kv_{m,1}(-1)^k.
$$
Since $A_mX_mv_{m,1}=v_{m,1}$ and $u_m^TX_m=u_m^T$, we have
$$
u_m^T A_m^{n-1}v_{m,1}=\sum\limits_{k=0}^{n}u_m^T A_m^{n-k}v_{m,1}u_m^T\bar{A}_m^kv_{m,1}(-1)^k.
$$
This is just
$$\overline{T}(n-1,m)=\sum^{n}_{k=0} (-1)^{k}\overline{T}(n-k, m)\overline{T}(k, m-1)=\sum^{n}_{k=0} (-1)^{n-k}\overline{T}(k, m)\overline{T}(n-k, m-1).$$
\end{proof}

\begin{proof}[First Proof of Theorem \ref{theo:Bona}]
It suffices to show that
\begin{equation}\label{eq:BonaTheo}
F(m,x)=\frac{1}{-x+F(m-1,-x)},
\end{equation}
which is equivalent to
$$xF(m,x)=F(m,x)F(m-1,-x)-1.$$
Put in the formula
$$F(m,x)=\sum_{k=0}^\infty T(k,m) x^k =\sum_{k=0}^\infty \bar{T}(k,m+1) x^k,$$
and equate coefficient of $x^n$, we need to prove
$$\bar{T}(n-1,m+1)= \sum_{k=0}^n \bar{T}(k,m+1) (-1)^(n-k) \bar{T}(n-k,m),   $$
which is just Corollary \ref{cor:Tnm} (replacing $m$ by $m+1$). This completes the proof.
\end{proof}

\def\bT{\overline{T}}
Define $\bT^{\a,\b}(x,m)=\sum\limits_{n=0}^{\infty}\a^TA_m^nx^n\b$.
Then $$\bT^{u_m,v_{m,1}}(x,m) = \sum\limits_{n=0}^{\infty} u_m^T A_m^n v_{m,1} x^n = F(m-1,x). $$

Equation \eqref{eq:alphabeta-AT} in Theorem \ref{theo:base1} can be written in the form of generating function of $x$ as follows.
\begin{theo}\label{theo:generatingfunction}
For a given positive integer $m$, we have
\begin{align}\label{eq:Tx-alphabeta}
&T^{\a,X_m\b}(x,m)-T^{\overline{X_m\a},\overline{\b}}(-x,m-1)+\alpha_1T^{\bar{u}_m,\bar{\b}}(-x,m-1)-\beta_mT^{\a,v_{m,1}}(x,m)\\
&=T^{\a,v_{m,1}}(x,m)T^{\bar{u}_m,\bar{\b}}(-x,m-1).\nonumber
\end{align}
\end{theo}
\begin{proof}
We multiply $x^n$ on both sides of Equation \eqref{eq:alphabeta-AT} of Theorem \ref{theo:base1} and take sum with respect to $n$ from 0 to infinity. We get
\begin{scriptsize}
\begin{equation}\nonumber
\sum\limits_{n=0}^{\infty}(\a^T A_m^nX_m\b-(-1)^{n}\a^TX_m\bar{A}_m^n\b+(-1)^{n}\alpha_1u_m^T\bar{A}_m^n\b-\a^T A_m^nv_{m,1}\beta_m)x^n=\sum\limits_{n=0}^{\infty}(\sum\limits_{k=0}^{n}\a^T A_m^{n-k}v_{m,1}u_m^T\bar{A}_m^k\b(-1)^k)x^n.
\end{equation}
\end{scriptsize}
In terms of generating functions, this is just \eqref{eq:Tx-alphabeta}.
\end{proof}

Let $\a=u_m$ and $\beta_m=0$ in Equation \eqref{eq:Tx-alphabeta} of Theorem \ref{theo:generatingfunction}. We have the following corollary.
\begin{cor}\label{cor:alphabeta}
For a given positive integer $m$, we have
$$T^{u_m,(0,\beta_{m-1},...,\beta_1)^T}(x,m)=T^{u_m,v_{m,1}}(x,m)T^{\bar{u},(\beta_1,\beta_{2},...,\beta_{m-1})^T}(-x,m-1).$$
\end{cor}

Let $\b=v_{m,i}$ with $i=1,...,m-1$ in Corollary \ref{cor:alphabeta}. We get the following corollary.
\begin{cor}
For a given positive integer $m$, we have
$$T^{u_m,v_{m,m-i+1}}(x,m)=T^{u_m,v_{m,1}}(x,m)T^{\bar{u},\bar{v}_{m-1,i}}(-x,m-1),$$ where $i=1,...,m-1$ and $\bar{v}_{m-1,i}$ is the $m-1$ dimensional vector obtained by removing the last component of $v_{m,i}$.
\end{cor}

\section{Some results about $A_n$}\label{sec:result-An}
From now on, we will consider $A_n$ instead of $A_m$ in last section.

\subsection{The Characteristic polynomial of $A_n$}

In this subsection, our main goal is to prove Theorem \ref{theo:fAn} and Conjecture \ref{conj:fAn}.

Recall that $A_n$ is the unit primitive matrix of order $n$. We need to introduce two closely related matrix
$T_n$ and $T'_n$ and their relation to the CPS $U_n(x)$.
Let $T_n'$ be an $n\times n$ matrix with ${(T_n')}_{ij}=\chi(|i-j|=1)$.
Damianou give the following result.
\begin{theo}\label{theo:fTn-Un-relation}
(Damianou \cite{fTn-Chebyshev}). For a given positive integer $n$, we have $f_{T_n'}(2x)=U_n(x)$.
\end{theo}

Let $T_n$ be the $n\times n$ matrix with entries $(T_n)_{ij}=(T_n')_{i,j}+\chi((i,j)=(1,1)).$
We will use the following fact.
\begin{fact}\label{fact:An-Un}
(Jeffery \cite{Jeffery}). For $n\ge 2$, we have $A_n=U_{n-1}(\frac{T_n}{2})$.
\end{fact}

We have the following conclusion.
\begin{lem}\label{lem:An-Tn-CP-equation}
For a given positive integer $n$, we have

\begin{itemize}
  \item [(1)] $f_{T_n}(2x)=f_{T_n'}(2x)-f_{T_{n-1}'}(2x)=\frac{\sin[(n+1)\theta]-\sin n\theta}{\sin \theta}$, where $\cos\theta=x$;
  \item [(2)] $f_{T_n}(x)=xf_{T_{n-1}}(x)-f_{T_{n-2}}(x)$;
  \item [(3)] $f_{A_n}(x)=x^2f_{A_{n-2}}(x)-(-1)^{n+1}f_{A_{n-1}}(-x)$.
\end{itemize}
\end{lem}
\begin{proof}
\begin{itemize}
\item [(1)] By splitting the first column of the determinant, we obtain
\begin{align*}
  f_{T_n}(2x)&=|2xI_n-T_n|\\
&=\begin{vmatrix}
                        2x & -1 & 0   &\cdots & 0 &0 \\
                        -1 & 2x & -1  &\cdots & 0 &0\\
                         0 & -1 & 2x  &\cdots & 0 &0\\
                         \vdots & \vdots &\vdots  & & \vdots&\vdots \\
                         0 & 0  & 0   &\cdots &2x  &-1\\
                         0 & 0  & 0   &\cdots &-1 &2x
                      \end{vmatrix}+\begin{vmatrix}
                        -1 & -1 & 0   &\cdots & 0 &0 \\
                         0 & 2x & -1  &\cdots & 0 &0\\
                         0 & -1 & 2x  &\cdots & 0 &0\\
                         \vdots & \vdots &\vdots  & & \vdots&\vdots \\
                         0 & 0  & 0   &\cdots &2x  &-1\\
                         0 & 0  & 0   &\cdots &-1 &2x
                      \end{vmatrix}\\
&=f_{T_n'}(2x)-f_{T_{n-1}'}(2x)\\
&=\frac{\sin[(n+1)\theta]-\sin n\theta}{\sin \theta},
\end{align*}
where in the last step, we use Theorem \ref{theo:fTn-Un-relation} and $\cos\theta=x$.
\item [(2)] By expanding along the rightmost column, we obtain
\begin{align*}
  f_{T_n}(x)&=|xI_n-T_n|\\
            &=x\begin{vmatrix}
                        x-1 & -1 & 0   &\cdots & 0 &0 \\
                        -1 & x & -1  &\cdots & 0 &0\\
                         0 & -1 & x  &\cdots & 0 &0\\
                         \vdots & \vdots &\vdots  & & \vdots&\vdots \\
                         0 & 0  & 0   &\cdots &x  &-1\\
                         0 & 0  & 0   &\cdots &-1 &x
                      \end{vmatrix}+\begin{vmatrix}
                        x-1 & -1 & 0   &\cdots & 0 &0 \\
                        -1 & x & -1  &\cdots & 0 &0\\
                         0 & -1 & x  &\cdots & 0 &0\\
                         \vdots & \vdots &\vdots  & & \vdots&\vdots \\
                         0 & 0  & 0   &\cdots &x  &-1\\
                         0 & 0  & 0   &\cdots &-1 &x
                      \end{vmatrix}\\
            &=xf_{T_{n-1}}(x)-f_{T_{n-2}}(x).
\end{align*}

\item [(3)] The proof is almost identical to that of Lemma \ref{lem:P-relation}, so we omit it here.
\end{itemize}
\end{proof}

Comparing the recursions in Lemma \ref{lem:An-Tn-CP-equation} gives the following result.
\begin{lem}\label{lem:fAn-fTn}
For a given positive integer $n$, we have
\begin{equation}\label{eq:fAn-fTn}
f_{A_n}(x)=(-1)^{\lfloor \frac{n+1}{2}\rfloor}f_{T_n}\left((-1)^{n+1}\frac{1}{x}\right)x^n.
\end{equation}
\end{lem}
\begin{proof}
We prove by induction on $n$.

It is straightforward to check the lemma for $n=1,2$. Assume the lemma holds for all positive integers less than $n$.
Then for general $n$, we have
\begin{align*}
  &f_{A_n}(x)\\
  \text{(by Lemma \ref{lem:An-Tn-CP-equation}(3))}&=x^2f_{A_{n-2}}(x)-(-1)^nf_{A_{n-1}}(-x)\\
            &=x^2(-1)^{\lfloor \frac{n-1}{2}\rfloor}f_{T_{n-2}}\left(\frac{(-1)^{n-1}}{x}\right)x^{n-2}
            -(-1)^{\lfloor \frac{n+1}{2}\rfloor+n}f_{T_{n-1}}\left(\frac{(-1)^{n}}{-x}\right)x^{n-1}\\
            \text{(by Lemma \ref{lem:An-Tn-CP-equation}(2))}&=x^{n-1}(-1)^{\lfloor \frac{n+1}{2}\rfloor+n}\left(x(-1)^{n}f_{T_{n-2}}\left(\frac{(-1)^{n}}{x}\right)-f_{T_{n-1}}\left(\frac{(-1)^{n+1}}{x}\right)\right)\\
            &=x^{n}(-1)^{\lfloor \frac{n+1}{2}\rfloor}f_{T_n}\left((-1)^{n+1}\frac{1}{x}\right).
\end{align*}
This completes the induction.
\end{proof}

%

Combining Lemma \ref{lem:An-Tn-CP-equation} and Equation \eqref{eq:Un-equation} gives the following explicit formula.
\begin{lem}\label{lem:CP-fTn}
For a given positive integer $n$, we have
$$f_{T_n}(x)=\sum\limits_{k=0}^{n}(-1)^{\lfloor \frac{n+1}{2}\rfloor+\lfloor \frac{3k}{2} \rfloor-nk-k}\binom{\lfloor \frac{n+k}{2}\rfloor}{k}x^k.$$
\end{lem}
\begin{proof}
\begin{align*}
  f_{T_n}(x)&=f_{T_n'}(x)-f_{T_{n-1}'}(x)\\
  &=\sum\limits_{k=0}^{\lfloor \frac{n}{2} \rfloor}(-1)^k\frac{(n-k)!}{k!(n-2k)!}x^{n-2k}-\sum\limits_{k=0}^{\lfloor \frac{n-1}{2} \rfloor}(-1)^k\frac{(n-1-k)!}{k!(n-1-2k)!}x^{n-1-2k}\\
  &=\left\{
\begin{aligned}
\sum\limits_{i\in n_e}(-1)^{\frac{n-i}{2}}\binom{\frac{n+i}{2}}{i}x^{i}+\sum\limits_{i\in n_o}(-1)^{\frac{n-i-1}{2}}\binom{\frac{n+i-1}{2}}{i}x^{i}&,\ \  \text{$n$ \ is even,} \\
\sum\limits_{i\in n_o}(-1)^{\frac{n-i}{2}}\binom{\frac{n+i}{2}}{i}x^{i}+\sum\limits_{i\in n_e}(-1)^{\frac{n-i-1}{2}}\binom{\frac{n+i-1}{2}}{i}x^{i}&,\ \  \text{$n$ \ is odd.} \\
\end{aligned}
\right.\\
  &=\sum\limits_{k=0}^{n}(-1)^{\lfloor \frac{n+1}{2}\rfloor+\lfloor \frac{3k}{2} \rfloor-nk-k}\binom{\lfloor \frac{n+k}{2}\rfloor}{k}x^k,
\end{align*}
where $n_o:=\{i:0\le i \le n, i \ \text{is\  odd} \}$ and $n_e:=\{i: 0\le i \le n, i \ \text{is \ even}\}$.
\end{proof}

By Lemmas \ref{lem:fAn-fTn} and \ref{lem:CP-fTn}, we get the explicit formula of $f_{A_n}(x)$ as follows.
\begin{lem}\label{theo:fAn-Characteristic}
For a given positive integer $n$, we have
$$f_{A_n}(x)=\sum\limits_{k=0}^{n}(-1)^{\lfloor\frac{3k}{2}\rfloor}\binom{\lfloor\frac{n+k}{2}\rfloor}{k}x^{n-k}.$$
\end{lem}

Compare with the definition of $P_n(x)$ in \eqref{e-Pnk}, we obtain the following result.
\begin{cor}\label{cor:Pnx-matrix}
For a given positive integer $n$, we have
$$P_n(x)=|I_n-xA_n|.$$
\end{cor}

Now, we give all solution of the equation $f_{T_n}(x)=0$ as follows.
\begin{lem}\label{theo:fTn-root}
For a given positive integer $n$, $T_n$ has $n$ distinct eigenvalues
$\delta_{n,j}=2\cos\frac{(2j-1)\pi}{2n+1}$ where $j=1,2,\dots,n$.
\end{lem}
\begin{proof}
By Equation \eqref{eq:Un-equation} and Lemma \ref{lem:An-Tn-CP-equation}, we get
$$f_{T_n}(x)=\frac{\sin[(n+1)\theta]-\sin n\theta}{\sin\theta}.$$
Here $\cos \theta = \frac{x}{2}$. Then for $1\le j\le n$, $\delta_{n,j}=2\cos \theta_{n,j}$ is a root of $f_{T_n}(x)$ if and only if
\begin{equation}\label{eq:fTn-theta}
\sin[(n+1)\theta_{n,j}]-\sin n\theta_{n,j}=0.
\end{equation}
The equality clearly holds when
$\theta_{n,j}=\frac{(2j-1)\pi}{2n+1}$, since $(n+1)\theta_{n,j}+n\theta_{n,j}=(2j-1)\pi$.
This completes the proof.
\end{proof}

Consequently, we can obtain all the roots of $f_{A_n}(x)$ by Lemmas \ref{lem:fAn-fTn} and \ref{theo:fTn-root}.
\begin{lem}\label{lem:fAn-Characteristiclambda}
For a given positive integer $n$, the unit-primitive matrix $A_n$ has $n$ distinct eigenvalues
$$\lambda_{n,j}=\frac{(-1)^{n+1}}{\delta_{n,j}}=\frac{(-1)^{n+1}}{2\cos\frac{(2j-1)\pi}{2n+1}}, $$
where  $j=1,2,\dots,n$.
\end{lem}

Recall Fact \ref{fact:An-Un} says that $A_n=U_{n-1}(\frac{T_n}{2})$. Then by Lemmas \ref{theo:fTn-root} and \ref{lem:fAn-Characteristiclambda}, we get the following result, which can also be obtained by direct calculation.
\begin{cor}\label{cor:Unlambda}
For $n\in \P$, we have
$$\{U_{n-1}(\delta_{n,j}/2),j=1,\dots,n\}=\{\lambda_{n,j},j=1,\dots,n\},$$ where $\delta_{n,j}=2\cos\frac{(2j-1)\pi}{2n+1}$ and $\lambda_{n,j}=\frac{(-1)^{n+1}}{\delta_{n,j}}$ with $j=1,2,\dots,n$.
\end{cor}

\begin{proof}[Proof of Theorem \ref{theo:fAn}]
The theorem is a direct consequence of Lemma \ref{lem:fAn-Characteristiclambda}.
\end{proof}

\begin{proof}[Proof of Conjecture \ref{conj:fAn}]

For part (i), we have
 \begin{align*}
                 f_{A_n}&=|xI_n-A_n| \\
                        &=x^n|I_n-\frac{1}{x}A_n|\\
                        \text{(by Corollary \ref{cor:Pnx-matrix})}&=x^nP_n(1/x).
               \end{align*}

Part (ii) follows by Lemma \ref{lem:fAn-Characteristiclambda} and Corollary \ref{cor:Unlambda}.
\end{proof}

%
%

We have the following result which is useful for asymptotic analysis.
\begin{lem}\label{theo:lambda-max}
For $n\in \P$, the eigenvalues $\lambda_{n,j}=\frac{(-1)^{n+1}}{2\cos\frac{(2j-1)\pi}{2n+1}}$
of $A_n$ are ordered as follows.
\begin{itemize}
  \item [(1)] If $n$ is odd, then $$\lambda_{n,{\lfloor\frac{n}{2}}\rfloor+2}<\cdots<\lambda_{n,{n}}<0<\lambda_{n,{1}}<\cdots<\lambda_{n,{\lfloor\frac{n}{2}\rfloor+1}},$$
 $$ \lambda_{n,j}>1 \Leftrightarrow    \frac{n+2}{3}<j<\frac{2n+3}{4}, \text{ and } \lambda_{n,j}<1 \Leftrightarrow \frac{2n+3}{4}<j<\frac{4n+5}{6};$$
  \item [(2)] If $n$ is even, then
  $$\lambda_{n,{\lfloor\frac{n}{2}\rfloor}}<\lambda_{n,{\lfloor\frac{n}{2}\rfloor-1}}<\cdots<\lambda_{n,{1}}<0<\lambda_{n,{n}}<\lambda_{n,{n-1}}<\cdots<\lambda_{n,{\lfloor\frac{n}{2}\rfloor+1}},$$
       $$\lambda_{n,j}>1 \Leftrightarrow     \frac{2n+3}{4}<j<\frac{4n+5}{6} ,\text{ and } \lambda_{n,j}<1 \Leftrightarrow \frac{n+2}{3}<j<\frac{2n+3}{4}.$$
\end{itemize}
Moreover, the spectral radius of $A_n$ is given by the formula
$$\max_{j=1}^n |\lambda_{n,j}| =\lambda_{n,{\lfloor \frac{n}{2}\rfloor+1}}= U_n(\cos(\frac{\pi}{2n+1}))=\frac{1}{2\sin(\frac{\pi}{2(2n+1)})}.$$
\end{lem}
\begin{proof}
Observe that the function $\cos(\theta)$ is monotonically decreasing in the interval $[0,\pi]$. The first statement is straightforward.

By comparing $|\lambda_{n,{\lfloor\frac{n}{2}}\rfloor+2}|$ with $|\lambda_{n,{\lfloor\frac{n}{2}\rfloor+1}}|$ when $n$ is odd, and similarly when $n$ is even,
we see that $\max_{j=1}^n |\lambda_{n,j}| =\lambda_{n,{\lfloor \frac{n}{2}\rfloor+1}}$ for all $n$.

Finally, direction computation gives
\begin{align*}
  U_n(\cos(\frac{\pi}{2n+1}))&=\frac{\sin[(n+1)\frac{\pi}{2n+1}]}{\sin(\frac{\pi}{2n+1})}
  =\frac{\cos[\frac{\pi}{2(2n+1)}]}{\sin(\frac{\pi}{2n+1})}=\frac{1}{2\sin(\frac{\pi}{2(2n+1)})};
\end{align*}
\begin{align*}
\lambda_{n,{\lfloor \frac{n}{2}\rfloor+1}}&=\frac{(-1)^{n+1}}{2\cos\frac{(2\lfloor \frac{n}{2}\rfloor+1)\pi}{2n+1}}\\
                                    &=\left\{
\begin{array}{ll}
\displaystyle\frac{-1}{2\cos\frac{(n+1)\pi}{2n+1}}=\frac{1}{2\sin(\frac{\pi}{2(2n+1)})},&  \text{when $n$ is even;}\\
\displaystyle\frac{1}{2\cos\frac{n\pi}{2n+1}}=\frac{1}{2\sin(\frac{\pi}{2(2n+1)})},&  \text{when $n$ is odd.}
\end{array}
\right.
\end{align*}
 This completes the proof.
\end{proof}

%
%
%

%
%
%

We finish this subsection by the following byproduct by direct computation.
\begin{prop}
For a given positive integer $n$, and $j=1,2,\dots,n$, we have
$$f_{T_n}(2\cos\frac{(2j-1)\pi}{n})=
\left\{
\begin{aligned}
-1,&\ j\neq \frac{n+1}{2},\\
(-1)^{n-1}(2n+1),&\ j=\frac{n+1}{2}.
\end{aligned}\right.
$$
\end{prop}

\subsection{Some results about $P_n(x)$ and the proof of Theorem \ref{theo:Fnxsin}}
In this subsection, our goal is to give some results about $P_n(x)$ and  prove Theorem \ref{theo:Fnxsin}.
Our starting point is the determinant formula $P_n(x)=\det(I_n -x A_n)$ proved in Corollary \ref{cor:Pnx-matrix}.

We use the following notations: $u_n=(1,1,\dots,1)^T \in \R^n$ and $v_{n,j}$ is the $j$-th unit column vector in $\R^n$;
for an order $n$ matrix $A$, we denote by $A^*$ the adjoint matrix of $A$. It is clear that
$v_{n,i}^T A^* v_{n,j}$ is just the $(i,j)$ entry of $A^*$, or the $(j,i)$-minor of $A$.

\begin{lem}\label{l-uAv}
For $n\in \P$, we have
$$u_{n}^T(I_{n}-xA_{n})^{*}v_{n,1}=P_{n-1}(-x). $$
\end{lem}
\begin{proof}
Observe that $u_{n}^T(I_{n}-xA_{n})^{*}v_{n,1}$ is just the sum of the first column entries of $(I_{n}-xA_{n})^{*}$.
Then it is standard to rewrite and compute as follows.
\begin{align*}
  u_{n}^T(I_{n}-xA_{n})^{*}v_{n,1}&=\begin{vmatrix}
                        1 & -x  & -x   &\cdots & -x & -x \\
                        1 & 1-x & -x   &\cdots & -x &  0 \\
                        1 & -x  & 1-x   &\cdots &  0 &  0 \\
                         \vdots & \vdots &\vdots  & & \vdots&\vdots \\
                        1 & -x  & 0    &\cdots &1   &0\\
                        1 & 0   & 0    &\cdots &0   &1
                      \end{vmatrix}\\
                      &\text{(Add the multiple of column 1 by $x$ to all the other columns)}\\
                      &=\begin{vmatrix}
                        1 & 0  & 0   &\cdots & 0 & 0 \\
                        1 & 1 & 0   &\cdots & 0 &  x \\
                        1 & 0  & 1   &\cdots &  x &  x \\
                         \vdots & \vdots &\vdots  & & \vdots&\vdots \\
                        1 & 0  & x    &\cdots &1+x   &x\\
                        1 & x   & x    &\cdots &x   &1+x
                      \end{vmatrix}=\begin{vmatrix}
                        1 & 0   &\cdots & 0 &  x \\
                        0  & 1   &\cdots &  x &  x \\
                         \vdots &\vdots  & & \vdots&\vdots \\
                        0  & x    &\cdots &1+x   &x\\
                         x   & x    &\cdots &x   &1+x
                      \end{vmatrix}\\
                      &=P_{n-1}(-x).
\end{align*}
\end{proof}

Similar determinant computation gives the following result as a byproduct.
\begin{prop}
For $n\ge 2$, we have
$$(I_n-xA_n)^*_{1,j}=\left \{
\begin{aligned}
P_{n-2}(x),\quad &j=1;\\
xP_{n-2j+1}(-x),\quad&1<j<\lfloor \frac{n+1}{2} \rfloor +1;\\
xP_{2j-n-2}(x),\quad&\lfloor \frac{n+1}{2} \rfloor<j\le n.
\end{aligned}\right.$$
\end{prop}

Combining the above with $(\lambda_{n,j}E-A_n)(\lambda_{n,j}E-A_n)^*=0$ for $1\le j\le n$, we obtain explicit
diagonalization of $A_n$ as follows.
\begin{cor}\label{cor:P-recursive}
For a given positive integer $n$ and $1\le j\le n$, the unit eigenvector of $A_n$
with respect to the eigenvalue $\lambda_{n,j}$ is
$${\boldmath\eta_{n,j}}=\frac{2}{\sqrt{2n+1}}(\eta_1(\theta_j),\eta_2(\theta_j), \dots,  \eta_n(\theta_j))^T$$
where $\theta_j=\frac{(2j-1)\pi}{2n+1}$, and
$$ \eta_k(\theta)=\left \{
\begin{aligned}
\frac{\cos\left(\frac{2n-3}{2}\theta\right)}{2\cos\theta},\quad &k=1;\\
(-1)^{k+1}\cos\left(\frac{2n-4k+3}{2}\theta\right) ,\quad&1<k\le n.
\end{aligned}\right.$$
Consequently, $U=(\eta_{n,1},\dots,\eta_{n_n})$ is an orthogonal matrix and $U^TA_nU=\diag(\lambda_{n,1},\dots, \lambda_{n,n})$.
\end{cor}

The following result follows by the explicit formula of $P_n(x)$ in \eqref{e-Pnk},
but we will give a determinant proof.
\begin{lem}\label{lem:P-relation}
For $n\in \P$, we have $P_{n+1}(x)=-xP_n(-x)+P_{n-1}(x).$
\end{lem}
\begin{proof}
We have
\begin{align*}
  P_{n+1}(x)&=|I_{n+1}-xA_{n+1}|   \qquad  \text{(by splitting
   column 1)} \\
         &=\underbrace{\begin{vmatrix}
                        -x &       -x &        -x  &   \cdots &    -x &  -x \\
                         -x &      1-x &        -x  &   \cdots &    -x &   0\\
                         -x &       -x &       1-x  &   \cdots &     0 &   0\\
                     \vdots &   \vdots &    \vdots  &          & \vdots&\vdots \\
                         -x &       -x  &        0   &   \cdots &    1 &  0\\
                         -x &       0  &        0   &   \cdots &     0 &  1
                      \end{vmatrix}}_{\textcircled{1}}+\underbrace{\begin{vmatrix}
                        \boxed{1} &       -x &        -x  &   \cdots &    -x &  -x \\
                         0 &      1-x &        -x  &   \cdots &    -x &   0\\
                         0 &       -x &       1-x  &   \cdots &     0 &   0\\
                     \vdots &   \vdots &    \vdots  &          & \vdots&\vdots \\
                         0 &       -x  &        0   &   \cdots &    1 &  0\\
                         0 &       0  &        0   &   \cdots &     0 &  \boxed{1}
                      \end{vmatrix}}_{\textcircled{2}}\\
          &=\underbrace{-xP_n(-x)}_{\textcircled{1}'}+\underbrace{P_{n-1}(x)}_{\textcircled{2}'},
\end{align*}
where $\textcircled{1}'$ is obtained from $\textcircled{1}$ by factoring out $-x$ from column 1 and then applying Lemma \ref{l-uAv}, and
$\textcircled{2}'$ is obtained from $\textcircled{2}$ by expanding along the two boxed entries in the determinant.
\end{proof}


A direct consequence of Lemma \ref{lem:P-relation} is the following.
\begin{cor}\label{cor:P-recursive}
For $n\in \P$, we have
$$\frac{P_n(-x)}{P_{n+1}(x)}=\frac{1}{-x+\frac{P_{n-1}(x)}{P_n(-x)}}.$$
\end{cor}

Now we give explicit formula of $F(n,x)$ and reprove Theorem \ref{theo:Bona}.
\begin{theo}\label{theo:Fny-Pn}
For $n\in \N$, we have $F(n,x)=\frac{P_{n}(-x)}{P_{n+1}(x)}$. Consequently by Corollary \ref{cor:P-recursive},
the continued fraction formulas in Theorem \ref{theo:Bona} hold true.
\end{theo}
\begin{proof}
\begin{align*}
  F(n,x)&=\sum\limits_{k=0}^{\infty}WL_k(n)x^k=\sum\limits_{k=0}^{\infty}(u_{n+1}^TA_{n+1}^ku_{n+1})x^{k+1}+1\\
        &=\sum\limits_{k=0}^{\infty}(u_{n+1}^TA_{n+1}^kv_{n+1,1})x^{k}=u_{n+1}^T(I_{n+1}-xA_{n+1})^{-1}v_{n+1,1}\\
        &=\frac{u_{n+1}^T(I_{n+1}-xA_{n+1})^{*}v_{n+1,1}}{|I_{n+1}-xA_{n+1}|}\\
        \text{(by Corollary \ref{cor:Pnx-matrix})}&=\frac{u_{n+1}^T(I_{n+1}-xA_{n+1})^{*}v_{n+1,1}}{P_{n+1}(x)}\\
       \text{(by Lemma \ref{l-uAv})} &=\frac{P_{n}(-x)}{P_{n+1}(x)}.
\end{align*}
\end{proof}

Next we derive a trig representation of $P_n(x)$ with the help of $T_n$.
\begin{lem}\label{lem:Pn-fTn}
Let $n\in \P$. Then we have
$$P_n(x)=(-1)^{\binom{n+1}{2}}f_{T_n}\left((-1)^{n+1}x\right).$$
Thus $P_n(x)$ has $n$ distinct roots $\frac{(-1)^{n+1}}{\delta_{n,j}}, \ 1\le j \le n$, where $\delta_{n,j}=2\cos\frac{(2j-1)\pi}{2n+1}$.
\end{lem}
\begin{proof}
By Lemma \ref{lem:CP-fTn}, we have
\begin{align*}
(-1)^{\binom{n+1}{2}}f_{T_n}((-1)^{n+1}x)&=\sum\limits_{k=0}^{n}(-1)^{\lfloor \frac{n+1}{2}\rfloor+\lfloor \frac{3k}{2} \rfloor-nk-k+\binom{n+1}{2}}\binom{\lfloor \frac{n+k}{2}\rfloor}{k}(-x)^{(n+1)k}\\
&=\sum\limits_{k=0}^{n}(-1)^{\lfloor \frac{3k}{2} \rfloor}\binom{\lfloor \frac{n+k}{2}\rfloor}{k}x^{k}=P_n(x).
\end{align*}
By Theorem \ref{theo:fTn-root}, $P_n(x)$ has $n$ distinct roots $\frac{(-1)^{n+1}}{\delta_{n,j}}$ as stated in the lemma.
\end{proof}
Note that we can also prove this lemma by using the recursions in Lemmas \ref{lem:An-Tn-CP-equation} and \ref{lem:P-relation}.

\begin{lem}\label{lem:Pn-theta}
For $n\in \P$, we have
 $$P_n(x)=\frac{\sin (n+1)\theta-\sin n\theta}{\sin \theta}(-1)^{\binom{n+1}{2}}=\frac{\cos\left(\frac{2n+1}{2}\theta\right)}{\cos\frac{\theta}{2}}(-1)^{\binom{n+1}{2}},$$ where\ $\cos\theta= \frac{(-1)^{n+1}x}{2}$.
\end{lem}
\begin{proof}
By Lemma \ref{lem:Pn-fTn} and \ref{lem:An-Tn-CP-equation}, we have
\begin{align*}
P_n(x)&=(-1)^{\binom{n+1}{2}}f_{T_n}\left((-1)^{n+1}x\right)\\
      &=(-1)^{\binom{n+1}{2}}\left[U_n\left((-1)^{n+1}x/2\right)-U_{n-1}\left((-1)^{n+1}x/2\right)\right]\\
      &=\frac{\sin (n+1)\theta-\sin n\theta}{\sin \theta}(-1)^{\binom{n+1}{2}}\\
      &=\frac{2\cos\left(\frac{2n+1}{2}\theta\right)\sin\frac{\theta}{2}}{\sin\theta}(-1)^{\binom{n+1}{2}}\\
      &=\frac{\cos\left(\frac{2n+1}{2}\theta\right)}{\cos\frac{\theta}{2}}(-1)^{\binom{n+1}{2}}.
\end{align*}
\end{proof}


\begin{proof}[Proof of Theorem \ref{theo:Fnxsin}]
We have
\begin{align*}
F(n,x)
&=\frac{P_{n}(-x)}{P_{n+1}(x)} \quad (\text{by Corollary \ref{cor:P-recursive}})\\
&=(-1)^{n+1}\frac{\frac{\sin (n+1)\theta'-\sin n\theta'}{\sin \theta'}}{\frac{\sin (n+2)\theta-\sin (n+1)\theta}{\sin \theta}} \quad (\text{by\ Lemma\  \ref{lem:Pn-theta}})\\
      &=(-1)^{n+1}\frac{\sin (n+1) \theta-\sin n\theta}{\sin (n+2)\theta -\sin (n+1)\theta}\\
      &=(-1)^{n+1}\frac{\cos (\frac{2n+1}{2}\theta) }{\cos(\frac{2n+3}{2}\theta)},
\end{align*}
where\ $\cos\theta= \frac{(-1)^{n+2}x}{2}$\ and $\cos\theta'= \frac{(-1)^{n+1}(-x)}{2}=\cos\theta$ so that we choose $\theta=\theta'$.
\end{proof}

\begin{cor}\label{cor:N-infty}
For $n\in \P$, the partial fraction decomposition of $F(n-1,x)$ is:
\begin{equation}\label{eq:Fnx-PFD}
F(n-1,x)=\sum_{j=1}^{n}\frac{\gamma_{n,j}}{1-x_{n,j}x},
\end{equation}
where $x_{n,j}=\frac{(-1)^{n+1}}{2\cos \theta_j}$ with $\theta_j=\frac{\pi(2j-1)}{2n+1}$ for $j=1,\dots, n$.
Moreover, $\gamma_{n,j}=(-1)^{n+1}\frac{2\cos \theta_j \tan^2 \theta_j}{2n+1}$, and
asymptotically we have
$$WL_N(n-1) \sim \frac{\left(\csc\frac{\pi}{2(2n+1)}\right)^{N-1}\cot^2\frac{\pi}{2(2n+1)}}{2^{N-1}(2n+1)},\quad N\to \infty.$$
\end{cor}
\begin{proof}
The partial fraction \eqref{eq:Fnx-PFD} follows by Lemmas \ref{theo:Fny-Pn} and \ref{lem:Pn-fTn}, where
$x_{n,j}= (-1)^{n+1}\delta_{n,j}=\frac{(-1)^{n+1}}{2\cos \theta_j}$,
and $\gamma_{n,j}$ are constants for all $j=1,\dots, n$.

For a particular $j$, the constant $\gamma_{n,j}$ can be computed by using L'Hospital's rule, where
we need the relation $2\cos\theta =(-1)^{n+1} x$ and $\mathrm{d}x=2(-1)^{n+2}\sin\theta \mathrm{d}\theta$.
\begin{align*}
   \gamma_{n,j} &= \frac{P_{n-1}(-x)}{P_n(x)} (1-x_{n,j}x)\Big|_{x=\frac{1}{x_{n,j}}} \\
   &=(-1)^n \frac{\cos (\frac{2n-1}{2}\theta) }{\cos(\frac{2n+1}{2}\theta)}(1-x_{n,j}x)\Big|_{\theta=\theta_j}\\
   &=(-1)^n\frac{\cos (\frac{2n-1}{2}\theta) }{-\frac{2n+1}{2}\sin(\frac{2n+1}{2}\theta)\mathrm{d}\theta}(-x_{n,j}\mathrm{d}x)\Big|_{\theta=\theta_j}\\
   &=(-1)^{n+1}\frac{\cos (\frac{2n-1}{2}\theta_j) \tan \theta_j}{\frac{2n+1}{2}\sin(\frac{2n+1}{2}\theta_j)}\\
   &=(-1)^{n+1}\frac{2\sin \theta_j \tan \theta_j}{2n+1}=(-1)^{n+1}\frac{2\cos \theta_j \tan^2 \theta_j}{2n+1}.
\end{align*}

By Theorem \ref{theo:lambda-max}, $\max\{x_{n,j},j=1,2,\dots,n\}=x_{n,\lfloor \frac{n}{2} \rfloor+1}$, denote by $2\cos(\tau)$.
By the fact $\cos(\tau)=\cos\frac{\pi(2\lfloor\frac{n}{2}\rfloor+1)}{2n+1} =(-1)^{n+1}\sin\frac{\pi}{2(2n+1)}$ and $ \sin(\tau)=\cos\frac{\pi}{2(2n+1)} $, we have
$$
WL_N(n-1) \sim \gamma_{n,\lfloor \frac{n}{2} \rfloor+1} x_{n,\lfloor \frac{n}{2} \rfloor+1}^N=
\frac{\left(\csc\frac{\pi}{2(2n+1)}\right)^{N-1}\cot^2\frac{\pi}{2(2n+1)}}{2^{N-1}(2n+1)},\quad N\to \infty.$$
\end{proof}

In particular, it is an easy exercise in calculus to show that
$$ T({n-1,n-1})=WL_{n-1}(n-1) \sim \frac{2^{n+1}n^{n-1}\sqrt{e}}{\pi^{n}},   n\to \infty.$$

\begin{rem}
Let $C_k$ be the graph of a $k$-cycle with $k\ge 3$. It was shown in \cite{Bona} that $WC_k(n)=\textrm{trace}(A_{n+1}^k)$, and
$$CF(n,x)=\sum_{k\ge 1} WC_k(n)x^k = \left\lfloor\frac{n+1}{2}\right \rfloor x - \frac{x \frac{d}{dx} P_{n+1}(x)}{P_{n+1}(x)}.$$
Then we can obtain more explicit formula of $CF(n,x)$. Moreover, we have
$$ WC_N(n-1) \sim  \left(\frac{(-1)^{(n+1)}}{2\cos \tau}\right)^{N}=\left(\frac{\csc\frac{\pi}{2(2n+1)}}{2}\right)^{N} ,\qquad N\to \infty.$$
\end{rem}

\section{Some results on the generating function of $T(n,m)$}\label{sec:result-T}
In this section, our main goal is to prove Conjectures \ref{conj:M-rho}-\ref{conj:M-limit}.

\subsection{Proof of Conjectures \ref{conj:M-rho}-\ref{conj:M-symmetric}}
If we assume Theorem \ref{t-4interpretations} holds true, then
$T(n,x)=\rho(L_n,x)=Ehr(L_n,x)$ has been studied.
For instance, \cite[Corollary 3.8]{Bona} asserts that
$$ \rho(L_n,x) =\frac{H(x)}{(1-x)^{n+1}}$$
for some symmetric polynomial $H(x)$ of degree at most $n$. Here by saying that $H(x)$ is \emph{symmetric}, we mean
that $x^d H(1/x)=H(x)$ if $\deg H(x)=d$. Such $H(x)$
is also called \emph{palindromic} in some literature. Indeed Bona et. al. gave a description of
$\rho(G,x)$ when $G$ is a simple bipartite graph in \cite[Theorem 3.6]{Bona}. But they seemed
to be unaware of Conjectures \ref{conj:M-rho}-\ref{conj:M-symmetric} and didn't study $\deg H(x)$.

Lee et. al. \cite{Lee} studied the Ehrhart series of graph polytopes and claimed that Conjectures \ref{conj:M-rho}-\ref{conj:M-symmetric}
are consequences of their Theorem 4 by showing that $P(L_n)$ is an $n$-dimensional regular positive
reflexive polytope with parameter $k=2$ (see the paper for the concepts).
Their claim is true, but need to modular Theorem \ref{t-4interpretations}.

The two conjectures can also be attacked by the magic labelling model. Firstly,
Stanley \cite{Stanley-reciprocityTheorem} studied magic labellings of general pseudo graphs and give a characterization of $\mathcal{M}(G,x)$.
In particular, $\mathcal{M}(G,x)$ has denominator $(1-x)^{n+1}$ if the graph obtained by removing all loops from $G$ is bipartite.
Since removing all loops from $\tilde{L}_n$ is just $L_{n+1}$, $\mathcal{M}(\tilde{L}_n)$ is of the desired form $\frac{H(x)}{(1-x)^{n+1}}$.

For details of the numerator $H(x)$, we would like to introduce Stanley's Reciprocity Theorem, which is a more general result and seems easier to use.
We basically follow \cite[Chapter 4.6]{EC1}.

\def\nul{\mathbf{nul}}

Let $A$ be an $r$ by $n$ matrix with integer entries. The null space of $A$ is denoted $\nul(A)=\{\alpha\in \R^n: A\alpha=0 \}$,
which is also the solution space $\nul(\Phi)$ of the homogeneous linear Diophantine system $\Phi: \ A \x=\mathbf{0}$.
The $\N$-solution set $\nul(A)\cap \N^n$ forms an additive monoid (semigroup with a unit).
Stanly's Reciprocity Theorem \cite{Stanley-reciprocityTheorem}[Theorem 4.1]
gives a nice connection between the following two
generating functions:
$$E(\x)=\sum\limits_{(\alpha_1,...,\alpha_n)\in \nul(A)\cap \N^n}x_1^{\alpha_1}x_2^{\alpha_2}\cdots x_n^{\alpha_n},\ \
\bar{E}(\x)=\sum\limits_{(\alpha_1,...,\alpha_n)\in \nul(A)\cap \P^n}x_1^{\alpha_1}x_2^{\alpha_2}\cdots x_n^{\alpha_n}.$$
 For further references and generalizations, see \cite{Xin-reciprocityTheorem}.
\begin{theo}\label{theo:ReciprocityStanley}
(Stanley's Reciprocity Theorem). Let $A$ be an $r$ by $n$ integral matrix of full rank $r$. If there is at least one $\alpha \in \P^n \cap \nul(A)$, then we have as rational functions
$$\bar{E}(x_1,...,x_n)=(-1)^{n-r} E(x_1^{-1},...,x_n^{-1}).$$
\end{theo}

\begin{proof}[Sketched Proof of Conjectures \ref{conj:M-rho}-\ref{conj:M-symmetric}]
By Theorem \ref{t-4interpretations} and the discussion above, we have shown
$$T(n,x)=M(\tilde{L}_n,x)=\frac{H(x)}{(1-x)^{n+1}}.$$
To prove Conjectures \ref{conj:M-rho}-\ref{conj:M-symmetric}, it is sufficient to show that
$$ x^3 \mathcal{M}(\tilde{L}_n,x) =(-1)^{n+1} \mathcal{M}(\tilde{L}_n,x^{-1}).$$
This will be explained by Stanley's Reciprocity Theorem.

Firstly, the system $\Phi$ of homogeneous linear Diophantine equations is given by
$$\Phi:\ m_{i-1}+m_{i}+\ell_{i}-s=0, \quad 1\le i\le n+1, \text{ set } m_0=m_{n+1}=0.$$
See Figure \ref{fig:L3L3tilde} for the $n=3$ case.
Secondly, the generating function $E(\x)$ refers to
$$ E(\mathbf{y},\mathbf{z},x)=\sum_{(m_1,\dots,m_n,\ell_1,\dots,\ell_{n+1},s)\in \nul(\Phi) \cap \N^{2n+2}} y_1^{m_1}\cdots y_n^{m_n} z_1^{\ell_1}\cdots z_{n+1}^{\ell_{n+1}} x^s,$$
and similarly for $\bar{E}(\mathbf{y},\mathbf{z},x)$. It is clear that $\mathcal{M}(\tilde{L}_n,x)=E(1,1,\dots, 1, x)$.
Thirdly, the vector $\eta=(1,\dots, 1, 3)$ belongs to $\P^{2n+2}\cap \nul(\Phi)$ so that Stanley's Reciprocity Theorem applies.
Finally, observe that $\alpha \in \P^{2n+2}\cap \nul(\Phi) \Leftrightarrow \alpha -\eta \in \N^{2n+2}\cap \nul(\Phi)$.
It follows that $\bar E(\mathbf{y},\mathbf{z},x)=y_1\cdots y_n z_1\cdots z_{n+1} x^3  E(\mathbf{y},\mathbf{z},x)$, and consequently
$$ x^3 \mathcal{M}(\tilde{L}_n,x) =(-1)^{n+1} \mathcal{M}(\tilde{L}_n,x^{-1}),$$
as desired.
\end{proof}

It is worth mentioning that Chen et al \cite{Chen} proved that $H(x)$ is unimodal.

%


\subsection{Proof of Conjecture \ref{conj:M-equation}}
Though $M_{n-2-i,i}$ is defined using the row generating function $T(n,x)$,
the proof of Conjecture \ref{conj:M-equation} relies on the column generating function $T(x,m)=F(m,x)$.

Firstly, we need to represent $M_{n-2-i,i}$ in terms of $T(n,m)$.
\begin{lem}\label{lem:M-T-relation}
For $m,n\in \N$, we have
\begin{equation}\label{eq:T-Mn}
M_{n-2-m,m}=\sum\limits_{k=0}^{m}\binom{n+1}{k}T(n,m-k)(-1)^k.
\end{equation}
Note that $M_{-i,m}=0$ for $i\ge 0$.
\end{lem}
\begin{proof}
The lemma follows by comparing the coefficients of $x^m$ on both sides of
 $$ (1-x)^{n+1} \sum_{k\ge 0} T(n,k) x^k = \sum\limits_{i=0}^{n-2}M_{n-2-i,i}x^i.$$
\end{proof}


Define $M_m(x)=\sum\limits_{n=0}^{\infty}M_{n,m}x^n$. Then, we get the following lemma.
\begin{lem}\label{lem:M-lemma}
For $m,n\in \N$ and $n\ge 2$, we have
$$M_m(x)=x^{-m-2}\sum\limits_{n=0}^{\infty}\left( \sum\limits_{k=0}^{m}\binom{n+1}{k}T(n,m-k)(-1)^k \right)x^n.$$
\end{lem}
\begin{proof}
By Lemma \ref{lem:M-T-relation}, we get
 $$\sum\limits_{n=0}^{\infty}M_{n-2-m,m}x^{n}=\sum\limits_{n=0}^{\infty}\left(\sum\limits_{k=0}^{m}\binom{n+1}{k}T(n,m-k)(-1)^k\right )x^n,$$
which is equivalent to the equation
  $$x^{2+m}M_m(x)=\sum\limits_{n=0}^{\infty}\left(\sum\limits_{k=0}^{m}\binom{n+1}{k}T(n,m-k)(-1)^k\right )x^n.$$
Then, we have $M_m(x)=x^{-m-2}\sum\limits_{n=0}^{\infty}\left( \sum\limits_{k=0}^{m}\binom{n+1}{k}T(n,m-k)(-1)^k \right)x^n.$
\end{proof}

In terms of $F(m,x)=T(x,m)=\sum\limits_{k=0}^{\infty}T(k,m)x^k,$  we obtain:
\begin{theo}\label{theo:M-equation-Dao}
For $m,n\in \N$ and $n\ge 2$, we have
$$M_m(x)=x^{-m-2}\sum\limits_{k=0}^{m}(-1)^k\frac{\mathrm{d}^{k}(xF(m-k,x))}{\mathrm{d}x}\frac{x^{k-1}}{k!}.$$
Or equivalently by Theorem \ref{theo:Fny-Pn}, we have
$$M_n(x)=x^{-n-2}\sum\limits_{k=0}^{n}(-1)^k\frac{\mathrm{d}^{k}(x\frac{P_{n-k}(-x)}{P_{n+1-k}(x)})}{\mathrm{d}x}\frac{x^{k-1}}{k!}.$$
\end{theo}
\begin{proof}
For $m,n\in \N$ and $n\ge 2$, and by Lemma \ref{lem:M-lemma}, we get
\begin{align*}
M_m(x)&=x^{-m-2}\sum\limits_{n=0}^{\infty}\left( \sum\limits_{k=0}^{m}\binom{n+1}{k}T(n,m-k)(-1)^k \right)x^n\\
      &=x^{-m-2} \sum\limits_{k=0}^{m}\left(\sum\limits_{n=0}^{\infty}\binom{n+1}{k}T(n,m-k)(-1)^k \right)x^n\\
      &=x^{-m-2}\sum\limits_{k=0}^{m}(-1)^k\frac{\mathrm{d}^{k}(F(m-k,x)x)}{\mathrm{d}x}\frac{x^{k-1}}{k!}.
\end{align*}
\end{proof}

To prove Conjecture \ref{conj:M-equation}, we need two more lemmas.
Firstly let us define the degree of a rational function $f(x)/g(x)$ to be $\deg f(x)/g(x) = \deg f(x)-\deg g(x)$.
\begin{lem}
Let $R_1(x)$ and $R_2(x)$ be any two nonzero rational functions. Then we have
\begin{enumerate}
  \item $\deg R_1(x)\pm R_2(x) \le \max( \deg R_1(x), \deg R_2(x))$, and we always have the equality if $\deg R_1(x)\neq \deg R_2(x)$;
  \item $\deg R_1(x)R_2(x) =\deg R_1(x) + \deg R_2(x)$;
  \item $\deg R_1'(x)\le \deg R_1(x)-1$. Consequently,  $\deg R_1^{(k)}(x)\le \deg R_1(x)-k$.
\end{enumerate}
\end{lem}
\begin{proof}
The first two items are obvious. The third one follows by the quotient rule of derivatives.
More precisely, write $R_1(x)=f(x)/g(x)$. Then
  $$R_1'(x)=\frac{f'(x)g(x)-f(x)g'(x) }{g(x)^2}.$$
Item 3 follows since $\deg f'(x)g(x)=\deg f(x)g'(x)=\deg f(x)+\deg g(x) -1$.
\end{proof}

\begin{lem}\label{lem:deg-xF}
For given $n\in \P$, we have $\deg xF(n,x)+1 \le -2$.
\end{lem}
\begin{proof}
By Theorem \ref{theo:Fny-Pn},
it suffice to show that
$\deg \left(x P_n(-x)+P_{n+1}(x)\right) \le n-1$.

By the explicit formula of $P_n(x)$ in , we have
\begin{align*}
  xP_{n}(-x)&+P_{n+1}(x)= x\sum\limits_{i=0}^{n}(-1)^{\lfloor\frac{3i}{2}\rfloor}\binom{\lfloor\frac{n+i}{2}\rfloor}{i}(-x)^i +\sum\limits_{i=0}^{n+1}(-1)^{\lfloor\frac{3i}{2}\rfloor}\binom{\lfloor\frac{n+1+i}{2}\rfloor}{i}x^i\\
  &=\left((-1)^{\lfloor\frac{3n}{2}\rfloor+n}+(-1)^{\lfloor\frac{3n+3}{2}\rfloor}
  \right)  x^{n+1}+\left((-1)^{\lfloor\frac{3n-3}{2}\rfloor+n-1}+(-1)^{\lfloor\frac{3n}{2}\rfloor}
   \right) x^n +Q_{n-1}(x)\\
  &= Q_{n-1}(x),
\end{align*}
where $Q_{n-1}(x)$ is a polynomial of degree at most $n-1$. This completes the proof.
\end{proof}

\begin{proof}[Proof of Conjecture \ref{conj:M-equation}]
By Theorem \ref{theo:M-equation-Dao}, we need to analyze the property of each summand.

For the term for $k=0$, $xF(n,x)\cdot x^{-1}$ has degree $-1$ and denominator $P_{n+1}(x)$.

For each term for $k$ with $1\le k\le n$, $(xF(n-k,x))^{(k)}$ clearly has denominator $P_{n-k+1}(x)^{k+1}$.
By Lemma \ref{lem:deg-xF}, it has
degree
$$ \deg \frac{x^k}{k!} (xF(n-k,x))^{(k)}=\deg x^k (xF(n-k,x)+1)^{(k)} \le -2.$$
Thus
$$ \deg  \sum\limits_{k=1}^{m}(-1)^k\frac{\mathrm{d}^{k}(xF(n-k,x))}{\mathrm{d}x}\frac{x^{k-1}}{k!} \le -2.$$
It follows that
$\deg M_n(x) =-3-n$.

By taking common denominators, we see that
$$M_n(x)=\frac{G_n(x)}{(P_1(x))^{n+1} (P_2(x))^n  ... (P_n(x))^2  P_{n+1}(x)}$$
for some polynomial $G_n(x)$. This proved part (1).

Next, the degree of the displayed denominator of $M_n(x)$ is equal to
\begin{align*}
  \sum\limits_{i=1}^{n+1}i(n+2-i)  &=\frac{(n+1)(n+2)(n+3)}{6}.
\end{align*}
This proves part (3).

Finally,
$$\deg G(x)= \frac{(n+1)(n+2)(n+3)}{6}+\deg M_n(x)=\frac{(n-1)(n+3)(n+4)}{6}.$$
This proves part (2).
\end{proof}

\begin{rem}
From the proof, we see that in general $G_n(x)$ need not be coprime to the displayed denominator.
Indeed, the degree of the denominator of $M_7(x)$ in Equation \ref{eq:M-equation-conj} is 120, but
our computation by Maple shows that the degree of the denominator of $M_7(x)$ is 118. The reason is that $P_1(x)$ and $P_7(x)$ have a common factor $x-1$, where $P_1(x)=-x+1$ and $P_7(x)= \left( x-1 \right)  \left( {x}^{2}-x-1 \right)  \left( {x}^{4}+{x}^{3
}-4\,{x}^{2}-4\,x+1 \right)$.
\end{rem}

\subsection{Proof of Conjecture \ref{conj:M-example}}
By using Theorem \ref{theo:M-equation-Dao}, \texttt{Maple} can easily produce $M_n(x)$ for $n \le 11$?.
However the result becomes too complex to print here, so we only do the cases for $n=1,2,3,4$.
Conjecture \ref{conj:M-example} then follows by our computation.

\begin{exam}\label{exam:computingM0}
$M_0(x)=1/(1-x)$. This follows easily by $M_{n,0}=1$ for $n\in \N$.
\end{exam}

\begin{exam}
$M_1(x)=\frac{1}{(1-x)^2(1-x-x^2)}$.
\end{exam}
\begin{proof}
By Theorem \ref{theo:M-equation-Dao}, we get
\begin{align*}
M_1(x) &=x^{-3}\sum\limits_{k=0}^{1}(-1)^k\frac{\mathrm{d}^{k}(F(1-k,x)x)}{\mathrm{d}x}\frac{x^{k-1}}{k!}\\
    &=x^{-3}\left(F(1,x)-\frac{\mathrm{d}^{1}(F(0,x)x)}{\mathrm{d}x}\right)\\
    &=x^{-3}\left(\frac{1+x} {1-x-{x}^{2}}-\frac{1} {\left({1-x}\right)^{2}}\right)\\
    &=\frac{1} {\left({1-x}\right)^{2}\left({1-x-{x}^{2}}\right)}.
\end{align*}
\end{proof}

\begin{exam}
$M_2(x)=\frac{1 - x^2 - x^3 - x^4 + x^5}{(1-x)^3(1-x-x^2)^2(1 - 2x - x^2 + x^3)}$.
\end{exam}
\begin{proof}
By Theorem \ref{theo:M-equation-Dao}, we get
\begin{align*}
M_2(x)&=x^{-4}\sum\limits_{k=0}^{2}(-1)^k\frac{\mathrm{d}^{k}(F(2-k,x)x)}{\mathrm{d}x}\frac{x^{k-1}}{k!}\\
    &=x^{-4}\left(\frac{1+x-{x}^{2}} {1-2\,x-{x}^{2}+{x}^{3}}-\left (\frac{\left({1+x}\right)x } {1-x-{x}^{2}}\right )'+\frac{x} {\left({1-x}\right)^{3}}\right)\\
    &=\frac{{1-{x}^{2}-{x}^{3}-{x}^{4}+{x}^{5}}} {\left({1-x}\right)^{3}\left({1-x-{x}^{2}}\right)^{2}\left({1-2\,x-{x}^{2}+{x}^{3}}\right)}.
\end{align*}
Here $(f(x))'$ means the derivative of $f(x)$ with respect to $x$.
\end{proof}

\begin{exam}
\begin{equation*}
M_3(x)=\frac{N_3(x)}{(1-x)^4(1-x-x^2)^3(1 - 2x - x^2 + x^3)^2(1 - 2x - 3x^2 + x^3 + x^4)},
\end{equation*}
where
\begin{multline*}
N_3(x)=1+x-6\,{x}^{2}-15\,{x}^{3}+21\,{x}^{4}+35\,{x}^{5}-13\,{x}^{6}-51\,{x}^{7}+3\,{x}^{8}\\
+21\,{x}^{9}+5\,{x}^{10}+{x}^{11}-5\,{x}^{12}-{x}^{13}+{x}^{14}.
\end{multline*}
\end{exam}
\begin{proof}
By Theorem \ref{theo:M-equation-Dao}, we get
\begin{align*}
&M_3(x) =\frac{1}{x^{5}}\sum\limits_{k=0}^{3}(-1)^k\frac{\mathrm{d}^{k}(F(3-k,x)x)}{\mathrm{d}x}\frac{x^{k-1}}{k!}\\
    &=\frac{1}{x^{5}}\left(\frac{1+2\,x-{x}^{2}-{x}^{3}} {1-2\,x-3\,{x}^{2}+{x}^{3}+{x}^{4}}-\frac{1+2\,x-4\,{x}^{2}+2\,{x}^{3}} {\left({1-2\,x-{x}^{2}+{x}^{3}}\right)^{2}}
    +\frac{x \left({2+3\,x+3\,{x}^{2}}\right)} {\left({1-x-{x}^{2}}\right)^{3}}-\frac{x^{2}} {\left({1-x}\right)^{4}}\right).
\end{align*}
Simplifying gives the desired formula.
\end{proof}

\begin{exam}\label{exam:computingM4}
\begin{equation*}
M_4(x) =\frac{N_4(x)}{(1-x)^5(1-x-x^2)^4(1 - 2x - x^2 + x^3)^3(P_4(x))^2 P_5(x)},
\end{equation*}
where
\begin{align*}
N_4(x)=1 +& 4x - 31x^2 - 67x^3 + 348x^4 + 418x^5 - 1893x^6 - 1084x^7
      + 4326x^8 \\
      &+ 4295x^9 - 7680x^{10} - 9172x^{11} + 9104x^{12} + 11627x^{13}- 5483x^{14} - 10773x^{15} \\
      &+ 1108x^{16} + 7255x^{17} + 315x^{18} - 3085x^{19} - 228x^{20} + 669x^{21}\\
      &+ 102x^{22} - 23x^{23} - 45x^{24} - 16x^{25} + 11x^{26}+ 2x^{27} - x^{28}.
\end{align*}
\end{exam}
\begin{proof}
By Theorem \ref{theo:M-equation-Dao}, we get
\begin{align*}
M_4(x) &=x^{-6}\sum\limits_{k=0}^{4}(-1)^k\frac{\mathrm{d}^{k}(F(4-k,x)x)}{\mathrm{d}x}\frac{x^{k-1}}{k!}\\
&=x^{-6}\left(\frac{1+2\,x-3\,{x}^{2}-{x}^{3}+{x}^{4}} {1-3\,x-3\,{x}^{2}+4\,{x}^{3}+{x}^{4}-{x}^{5}}-\frac{1+4\,x-4\,{x}^{2}-2\,{x}^{3}+4\,{x}^{4}+2\,{x}^{5}} {\left({1-2\,x-3\,{x}^{2}+{x}^{3}+{x}^{4}}\right)^{2}}\right) \\
    &+x^{-6}\left(\frac{x \left({3+3\,{x}^{2}-11\,{x}^{3}+9\,{x}^{4}-3\,{x}^{5}}\right)} {\left({1-2\,x-{x}^{2}+{x}^{3}}\right)^{3}}-\frac{x^{2}\left({3+8\,x+6\,{x}^{2}+4\,{x}^{3}}\right)} {\left({1-x-{x}^{2}}\right)^{4}}+\frac{x^{3}} {\left({1-x}\right)^{5}}\right).
\end{align*}
Simplifying gives the desired result.
\end{proof}

%

\subsection{Proof of Conjecture \ref{conj:M-limit}}
Our proof relies on Corollary \ref{cor:N-infty}, we first give a lemma.

By Equation \eqref{eq:T-Mn}, we have
 $$ M_{n-2-m,m}=\sum\limits_{k=0}^{m}\binom{n+1}{k}T(n,m-k)(-1)^k.$$
Furthermore, by Corollary \ref{cor:N-infty}, we can obtain the following result.
\begin{lem}
For $n-2\ge m$, we have
$$M_{n-2-m,m}\sim \frac{\left(\csc\frac{\pi}{2(2m+1)}\right)^{n-1}\cot^2\frac{\pi}{2(2m+1)}}{2^{n-1}(2m+1)},\quad n\to \infty.$$
\end{lem}

\begin{proof}[Proof of Conjecture \ref{conj:M-limit}]
We have
\begin{align*}
&\mathrm{lim}_{n\rightarrow\infty}\frac{M_{n+1,m}}{M_{n,m}} \\
&= \mathrm{lim}_{n\rightarrow\infty}\frac{\frac{\left(\csc\frac{\pi}{2(2m+1)}\right)^{n+m+2}\cot^2\frac{\pi}{2(2m+1)}}{2^{n}(2m+1)}}{\frac{\left(\csc\frac{\pi}{2(2m+1)}\right)^{n+m+1}\cot^2\frac{\pi}{2(2m+1)}}{2^{n-1}(2m+1)}}\\
&=\frac{1}{2\sin\frac{\pi}{2(2m+1)}}\\
&=\frac{(-1)^{m+1}}{2\cos(\frac{(2\lfloor \frac{m}{2}\rfloor+1)\pi}{2m+1})}\\
%
\text{(by Corollary \ref{theo:lambda-max})}&= U_m(\cos(\frac{\pi}{2m+1})).
\end{align*}
\end{proof}

\textbf{Acknowledgments:} The authors would like to thank B\'{o}na for helpful conversations.

\end{document}